\documentclass[a4paper]{article}
\usepackage[english]{babel}
\usepackage[utf8]{inputenc}
\usepackage[T1]{fontenc}
\usepackage{mathpazo}
\usepackage{amsfonts,amsmath,amssymb,amsthm}
\usepackage{graphicx}
\usepackage{epstopdf}
\usepackage{algorithmic}
\usepackage{xypic}
\usepackage{xcolor}

\newcommand{\N}{\mathbf N}

\newcommand{\ra}{\rightarrow}

\DeclareOldFontCommand{\rm}{\normalfont\rmfamily}{\mathrm}

\newtheorem{theorem}{Theorem}[section]
\newtheorem{corollary}[theorem]{Corollary}
\newtheorem{proposition}[theorem]{Proposition}

\newtheorem*{corollary*}{Corollary}

\theoremstyle{definition}
\newtheorem{definition}{Definition}[section]

\title{Metrics and stabilization in one parameter persistence\thanks{First author was funded by VR and G\"oran Gustafsson foundation. Second author was partly supported by a collaboration agreement between the University of Aberdeen and EPFL.}}

\author{Wojciech Chach\'olski\thanks{Mathematics Department, KTH, S-10044 Stockholm, Sweden.
		(\texttt{wojtek@kth.se}).}
	\and Henri Riihim\"aki \thanks{Mathematics Department, University of Aberdeen, Aberdeen AB243UE, Scotland, UK.
		(\texttt{henri.riihimaki@abdn.ac.uk}).}}

\usepackage{amsopn}

\begin{document}

\maketitle

\begin{abstract}
	We propose the use of persistent homology in a supervised way. We believe homological persistence is fundamentally not about decomposition theorems but a central role is played by a choice of metrics. Choosing a pseudometric between persistent vector spaces leads to a  model. Fitting this model is what we believe supervised homological persistence is.  We develop theory behind constructing such models and  we give evidence of the usefulness of this approach in concrete data analysis tasks. 
\end{abstract}

\section{Sense of geometry}\label{adfgafdghsfgh}
During the last decade  there has been a considerable increase in research focused on persistent homology. This has been fueled by an explosion of applications ranging from neuroscience \cite{neurons}, to vehicle tracking \cite{vehicles}, and the characterization of nanomaterials \cite{nanomat}, testifying to usefulness of homology to understand spaces described  by measurements and samplings. All these applications of persistent homology  have been   in principle exploratory in nature with some elements of learning based on persistent diagrams.  In fact, persistent homology can be regarded as  a generalization   to higher homologies   of clustering methods ($H_0$ persistence) that have been the core of exploratory data analysis for a long time.   Although exploratory tools are  important, the main research front in modern data science has shifted  from exploratory to supervised learning, due to even more spectacular applications of machine learning methods.  

Our aim in this paper is to explain how to use persistent homology  in a supervised way, allowing to optimize over various models for the observed homological information.
The focus is on studying the space of stable translations  from  homological  information into information that can be analysed through more basic operations such as counting and integration  enabling the use of statistical tools to its  outcomes. We show how  a pseudometric on the set $\text{Tame}([0,\infty),\text{Vec}_K)$ of tame $[0,\infty)$-parametrized $K$ vector spaces (see Section~\ref{sdfgdfghfgh}),  which is  a natural place where homological invariants of  data live, leads to  a stable translation. Every such pseudometric  gives  therefore a  model  for extracting information from persistent homology.
Fitting  this  model to the training data is what in our approach persistence based  supervised learning  is. In Section~\ref{contour_usage}  of the paper we  illustrate   that this strategy can indeed  lead to  improvements in  classification tasks. We consider two such tasks: distinguishing between random point processes on a unit square generated according to different distributions, and distinguishing between activities of ascending and descending stairs of 7 people based on the activity monitoring PAMAP2  data obtained from~\cite{PAMAP}.
By choosing a different model, in the first case the overall averaged accuracy improves from $73\%$ to $78\%$ and in the second case from $60\%$ to $65\%$. Our goal for this paper however is not   to  benchmark   our approach. 
Our goal is to  present  the proof of concept of the key ideas, explain their mathematical background, and  indicate  that  they   can lead to improvements  in analysing data. Discussing the effectiveness of this approach is planned for a sequel to this article.

Our models are built using a process  called  {\bf hierarchical stabilization of the rank} which assures the necessary stability requirement. It  builds  on work presented in~\cite{OliverThesis},~\cite{OliverWojtek}, and~\cite{MR3735858}.
The input to the process is a pseudometric $d$ on  the set $\text{Tame}([0,\infty),\text{Vec}_K)$. The output is a Lipschitz-continuous function $\widehat{\text{rank}}_d\colon \text{Tame}([0,\infty),\text{Vec}_K)\to {\mathcal M}$ where ${\mathcal M}$ is the space of Lebesgue measurable functions $[0,\infty)\to [0,\infty)$ in which   probability and statistical methods are well developed.  We think about this function as the  model associated to the pseudometric $d$. In this framework (supervised) persistence  analysis  is about identifying these pseudometrics $d$ for which structural properties of  the (training)  data are  reflected by the geometry of its image in ${\mathcal M}$ through the  function $\widehat{\text{rank}}_d$.  The strategy of looking for appropriate pseudometrics can only   work if we are able to parametrize explicitly  a rich subspace of  pseudometrics on $\text{Tame}([0,\infty),\text{Vec}_K)$.  Such  rich parametrizations  would enable the  use of for example stochastic gradient  descent techniques to search through the parameters for  suitable pseudometrics, which we intend to explore  in the mentioned sequel to this paper.  
This article builds on our discovery that such parametrizations are indeed possible using Lebesgue measurable functions $[0,\infty)\to (0,\infty)$ with positive values referred to as densities (see~\ref{sdfbhndfjhgh} and~\ref{shift type}). 

Parametrizing models for persistence analysis by pseudometrics should not be surprising. Discovering appropriate metrics and units to measure physical phenomena is essential in understanding these phenomena. Comparison and interpretation of observations should depend on the phenomena and the experiments they came from and not  simply just on their values. Different phenomena might require different comparison metrics.  We should not restrict ourselves to only  bottleneck or Wasserstein distances  to compare outcomes of persistence analysis of  diverse data sets obtained from a variety of different experiments. We should be able to choose metrics that fit particular experiments. Our goal is to  present  mathematical foundations of how to do it for outcomes of persistent homology. This fits also well with many recent studies (\cite{Bendich2016}, \cite{Hiraoka_amorphous}, \cite{XiaWei_protein}, \cite{Stolz2016}, \cite{XiaWei_CryoEM}) which challenge the traditional view in persistence that bars with long lifespans are of importance and smaller bars are to be considered as noise. These  studies show that also shorter  bars and their appearance in the filtration might carry important information. For example in \cite{Hiraoka_amorphous} the authors find that observed diffraction peaks of amorphous silica glass relate to small scale loops in the atomic configurations. It thus depends on the analysis at hand what is to be taken as  noise and we advocate that emphasizing the meaningful features is a question of choosing an appropriate metric on $\text{Tame}([0,\infty),\text{Vec}_K)$. 

\section{Hierarchical stabilization}\label{sec:hier_stab}

All the proofs of the propositions presented  in this  section are placed in  Appendix~\ref{adfdfsgfhd}.

A discrete invariant is  a function $R\colon T\to {\mathbf N}$ with values in the set of natural numbers
${\mathbf N}=\{0,1,\ldots\}$. We think about $T$ as a collection of data sets or objects that represent them. If $T$ consists of finite metric spaces for example, such an invariant might assign to a metric space in  $T$ the number of clusters obtained by applying some clustering algorithm. Our method of converting such a discrete invariant into a stable one, which we call {\bf hierarchical stabilization}, requires a choice of a pseudometric $d$ on $T$ and this is  the key step central to persistence analysis in our approach. Recall that a pseudometric is a function $d\colon T\times T\to [0,\infty]$ satisfying reflexivity $d(X,X)=0$, symmetry $d(X,Y)=d(Y,X)$ and triangular inequality $d(X,Y)+d(Y,Z)\geq d(X,Z)$ for any $X$, $Y$, and $Z$ in $T$. Here $[0,\infty]$ denotes the extended set of non-negative real numbers  including $\infty$ with the  standard arithmetic and order relation.  Its subset of real numbers is denoted by $[0,\infty)$.

Once a pseudometric $d$ on $T$ is chosen, for  $X$ in $T$, we  define $\widehat{R}_d(X)\colon [ 0,\infty)\to [ 0,\infty)$ to be the function given by the formula:
\[\widehat{R}_d(X)(t):=\text{min}\{R(Y)\ |\ d(X,Y)\leq t\}.\]
The number $\widehat{R}_d(X)(t)$ is the minimum among all the values $R$ takes on the   disk $B_d(X,t):=\{Y\ |\ d(X,Y)\leq  t\}$ around $X$ with radius $t$ with respect to the pseudometric $d$. 
This function is non-increasing with values in natural numbers and hence   Lebesgue measurable. Furthermore
there is $t$  such that, for  $s\geq t$ in $[0,\infty)$, the equality $\widehat{R}_d(X)(t)=\widehat{R}_d(X)(s)$ holds. 
This value  $\widehat{R}_d(X)(t)$ is  called the {\bf limit} of $\widehat{R}_d(X)$ and is denoted  by
$\text{lim}(\widehat{R}_d(X))$. Recall that  ${\mathcal M}$  denotes  the set of  Lebesgue measurable functions $[0,\infty)\to [0,\infty)$.

\begin{definition}\label{asdghsfghdfg}
	{\bf Hierarchical stabilization} of $R\colon T\to {\mathbf N}$ with respect to a pseudometric $d$ on $T$  is the  function $\widehat{R}_{d}\colon T\to{\mathcal M}$ that maps $X$ in $T$ to $\widehat{R}_d(X)$.
\end{definition}

The range ${\mathcal M}$ of $\widehat{R}_d$ has a much richer geometry than the set of natural numbers,  the range of $R$. For example, ${\mathcal M}$ has many interesting pseudometrics, among them  the standard \(L_p\)-metric (for $p\geq 1$) and  a so called   \textbf{interleaving} metric \(d_{\bowtie}\):
\begin{align*}
	L_p (f,g)&:=\left( \int_0^\infty|f(t)-g(t)|^p dt \right) ^{\frac{1}{p}}\\
	S& :=\{\epsilon \, | \, f(t) \ge g(t+\epsilon) \ \text{and} \ g(t) \ge f(t+\epsilon) \ \text{for all} \ t \in [0,\infty)\}\\
	d_{\bowtie} (f,g)&:=\begin{cases}
		\text{inf}(S), & \text{if $S$ is non-empty}\\
		\infty, & \text{otherwise.}
	\end{cases}
\end{align*}

The hierarchical stabilization satisfies the following Lipschitz properties:
\begin{proposition}\label{prop:Lipschitz}
	Let $d$ be a pseudometric on $T$, $R\colon T\to {\mathbf N}$  a function, and $p\geq 1$ a real number. Then,
	for  $X$ and $Y$ in $T$:
	\begin{enumerate}
		\item 
		$d(X,Y)\geq d_{\bowtie}\left(\widehat{R}_d(X),\widehat{R}_d(Y)\right).$
		\item  $c\, d(X,Y)^{\frac{1}{p}}\geq L_p \left(\widehat{R}_d(X),\widehat{R}_d(Y)\right),$
		where $c=\text{\rm max}\{\widehat{R}_d(X)(0),\widehat{R}_d(Y)(0)\} .$
	\end{enumerate}
\end{proposition}

We  think about the hierarchical stabilization as a process of converting a discrete invariant $R\colon T\to{\mathbf N}$ into a stable invariant $\widehat{R}_d\colon T\to{\mathcal M}$ whose  values are in  a space in which rich probability and statistical methods are  well developed. We take advantage of this in our examples in Sections~\ref{point_processes} and~\ref{FASGDFSG}. Different  pseudometrics  on $T$ lead to different invariants. In our framework persistence analysis  is about  identifying pseudometrics  on $T$ for which the associated  invariants  
reflect structural properties of $T$. The expectation  is  that some of these properties  should be reflected by  the  geometry of the image of $\widehat{R}_d$ in ${\mathcal M}$ described   by the  $L_p$  or interleaving metrics if  
an appropriate pseudometric  $d$ on $T$ is chosen.  We  refer to  the function 
$\widehat{R}_d\colon T\to{\mathcal M}$ also as the {\bf hierarchical stabilization model} of $T$ associated to the invariant $R$ and the pseudometric $d$.

In general there is a loss of information as  $\widehat{R}_d\colon T\to{\mathcal M}$  may map objects that we do not  intend to identify to the same function.  For  retaining more   information we are going to consider families of pseudometrics on $T$ and the induced   stabilizations. Let ${\mathcal M}_2$ denote the set of measurable functions of the form $[0,\infty)^2\to [0,\infty)$.

\begin{definition}\label{asfgdsfghs}
	Let $R\colon T\to  {\mathbf N}$ be a function and 
	$\{d_\alpha\}_{\alpha\in[0,\infty]}$ a sequence of pseudometrics  on  $T$ indexed by $[0,\infty]$.
	\begin{enumerate}
		\item  The sequence  $\{d_\alpha\}_{\alpha\in[0,\infty]}$ is called 
		{\bf non-decreasing} if $d_\alpha(X,Y)\leq d_\beta(X,Y)$
		for all $\alpha<\beta$ in $[0,\infty]$ and $X$, $Y$ in $T$.
		\item For $X$ in $T$, $\overline{R}(X)\colon [0,\infty)^2\to [0,\infty)$ is a function  defined as follows:
		\[\overline{R}(X)(\alpha,t):=\widehat{R}_{d_{\alpha}}(X)(t).\]
		\item    The sequence $\{d_\alpha\}_{\alpha\in[0,\infty]}$ is called  {\bf admissible} for $R$ if  
		$\overline{R}(X)\colon [0,\infty)^2\to [0,\infty)$  is Lebesgue measurable for all $X$ in $T$.
		\item Assume  $\{d_\alpha\}_{\alpha\in[0,\infty]}$  is  admissible for $R$.
		Then  the  function
		$\overline{R}\colon T\to {\mathcal M}_2$,
		mapping  $X$ in $T$ to $\overline{R}(X)\colon [0,\infty)^2\to [0,\infty)$, is called
		the {\bf hierarchical stabilization} of $R$ along the  sequence $\{d_\alpha\}_{\alpha\in[0,\infty]}$.
	\end{enumerate}
\end{definition}

Non-decreasing   sequences are  key examples of universally admissible sequences:

\begin{proposition}\label{sgdsths}

	A non-decreasing  sequence  $\{d_\alpha\}_{\alpha\in[0,\infty]}$ of pseudometrics on $T$ is admissible for any 
	$R\colon T\to  {\mathbf N}$.
\end{proposition}

Similarly to ${\mathcal M}$,   we consider following  pseudometrics on  ${\mathcal M}_2$, the  normalized $L_p$
(for $p\geq 1$) and the  interleaving metrics:
\begin{align*}
	\widehat{L_p}(f,g)&:=\lim_{a\to\infty}\frac{1}{a}\int_0^a \left(\int_0^\infty|f(\alpha,t)-g(\alpha,t)|^p dt \right)^{\frac{1}{p}} d\alpha  \\
	S& :=\left\{\epsilon \, | \, 
	\begin{array}{c}f(\alpha,t) \ge g(\alpha,t+\epsilon)\\
		g(\alpha,t) \ge f(\alpha,t+\epsilon) 
	\end{array}
	\ \text{for} \ (\alpha,t) \in [0,\infty)\times [0 ,\infty)\right\}\\
	d_{\bowtie} (f,g)&:=\begin{cases}
		\text{inf}(S), & \text{if $S$ is non-empty}\\
		\infty, & \text{otherwise.}
	\end{cases}
\end{align*}

The key parameter in the  hierarchical stabilization  is the choice of a pseudometric. It turns out that with respect to this key parameter  the  hierarchical stabilization along a sequence of pseudometrics $\{d_\alpha\}_{\alpha\in[0,\infty]}$ is also stable. Here is one manifestation of this stability:

\begin{proposition}\label{afgsfhg}
	Let $R\colon T\to {\mathbf N}$ be a function.   Assume $\{d_\alpha\}_{\alpha\in[0,\infty]}$ is a non-decreasing sequence  of pseudometrics on $T$.
	Let $\overline{R}\colon T\to {\mathcal M}_2$ be  the hierarchical stabilization of $R$   along this sequence.
	Then, for  $X$ and $Y$ in $T$:
	\begin{enumerate}
		\item 
		$d_{\infty}(X,Y)\geq d_{\bowtie}\left(\overline{R}(X),\overline{R}(Y)\right).$
		\item $c\, d_{\infty}(X,Y)^{\frac{1}{p}}\geq \widehat{L_p} \left(\overline{R}(X),\overline{R}(Y)\right),$
		where $c=\text{\rm max}\{\widehat{R}_{d_\infty}(X)(0),\widehat{R}_{d_\infty}(Y)(0)\}.$
	\end{enumerate}
\end{proposition}
The hierarchical stabilization process along   a non-decreasing  sequence   of pseudometrics on $T$   converts a discrete invariant $R\colon T\to {\mathbf N}$ into a stable invariant $\overline{R}\colon T\to {\mathcal M}_2$. We can now state our  key definition:

\begin{definition}\label{sdfgdghdgjn}
	Let $T$ be a category with its set of objects also denoted by  $T$.  A sequence of pseudometrics  $\{d_\alpha\}_{\alpha\in[0,\infty]}$  on $T$ is called {\bf ample} for  $R\colon T\to {\mathbf N}$ if  it is admissible  for $R$, and
	the hierarchical stabilization 
	$\overline{R}\colon T\to {\mathcal M}_2$  along this sequence has the following property:
	$X$ and $Y$ in $T$  are isomorphic if and only if $\overline{R}(X)=\overline{R}(Y)$.
\end{definition}

Let $T$ be a category. By definition  ample for $R\colon T\to {\mathbf N}$  sequences of pseudometrics on $T$ lead to stable embeddings of  isomorphism classes of objects in $T$  into ${\mathcal M}_2$. By choosing such an embedding, we can think about  $T$ as a subspace of  ${\mathcal M}_2$  in which rich probability and statistical methods are  well developed. Different  sequences which are ample for $R$ give different embeddings. Our expectation is that by choosing an appropriate such embedding, structural properties of $T$   relevant to a  data analysis task could be reflected by  the  geometry of the image of $\overline{R}$ in ${\mathcal M}_2$ described   by the  $L_p$  or interleaving metrics. The focus of this article is on $T$ being  the category $\text{Tame}([0,\infty),\text{Vec}_K)$ of tame $[0,\infty)$-parametrized $K$ vector spaces (also called one parameter tame persistence modules)  and 
$R\colon \text{Tame}([0,\infty),\text{Vec}_K)\to{\mathbf N}$ being the  minimal number of  generators or equivalently the number of bars in the bar decomposition.

\section{Formal homological persistence}\label{sdfgdfghfgh}
A typical input for  persistent  homology is a  collection $\{(X_i,d_i)\}$ of finite  pseudometric spaces. Since homological  persistence is about looking for homological features, each element of such  data needs to be transformed  into an object reflecting these features  more directly. The first step in this transformation is to convert 
the metric information into a simplicial complex parametrized by the poset  $[0,\infty)$ whose elements are referred to as scales in this context. In this paper  the Vietoris-Rips construction is used  for that purpose, which   at a scale $a\in [0,\infty)$  is a simplicial complex $\text{VR}_a(X_i,d_i)$ whose $k$-simplices are subsets of $X_i$ consisting of $k+1$ points which are pairwise at most distance $a$ from each other. For example, two elements  $p_1$ and $p_2$ in $X_i$ connect to an edge, or a 1-simplex, when $d_i(p_1,p_2) \le a$. If $a \le b$, then  $  \text{VR}_a(X_i,d_i) \subseteq \text{VR}_b(X_i,d_i) $. The obtained  filtration indexed by the poset $[0,\infty)$  is denoted by the  symbol $\text{VR}(X_i,d_i)$.

Note that $\text{VR}(X_i,d_i)$ does not add or forget any information about $(X_i,d_i)$ and hence is as complicated as  the metric space itself. Simplification is therefore necessary and this is the purpose of the second step in this transformation in which the $n$-th homology (with coefficients in a chosen field $K$) is applied to the simplicial complexes  in the Vietoris-Rips filtration. This results in  a functor $H_n (\text{VR}(X_i,d_i),K)\colon [0,\infty)\to \text{Vec}_K$  given by the linear functions  $ H_n (\text{VR}_a(X_i,d_i),K) \ra H_n (\text{VR}_b(X_i,d_i),K)$    induced by the inclusions  $\text{VR}_a(X_i,d_i) \allowbreak \subseteq \text{VR}_b(X_i,d_i)$ for  $a \leq b$  in $[0,\infty)$. A functor  of the form $V\colon  [0,\infty)\to \text{Vec}_K$ is also called a  {\bf $[0,\infty)$-parametrized $K$ vector space} and  the linear function $V_{a\leq b}\colon V_a\to V_b$, for  $a \leq b$  in $[0,\infty)$, is called a   {\bf transition function}. 

The functor  $H_n (\text{VR}(X_i,d_i),K)$ is not  an arbitrary  $[0,\infty)$-parametrized $K$ vector space. It satisfies additional two properties which  follow from  finiteness of  $X_i$:
\begin{definition}\label{agdshgffd}
	An $[0,\infty)$-parametrized $K$ vector space $V$ is called {\bf tame} if:
	\begin{enumerate}
		\item   the vector space $V_a$ is finite dimensional for   every  $a$ in $[0,\infty)$, 
		\item there are finitely many  $0<t_0<\cdots <t_k$ in $[0,\infty)$  such that, for $a\leq b$ in $[0,\infty)$, a transition function $V_{a\leq b}\colon V_{a}\to V_{b}$  may fail  to be an isomorphism only if  $a<t_i\leq b$ for some $i$. 
	\end{enumerate}
\end{definition}

For example constant functors are tame, in particular the $0$ functor. If $X$ is   finite, then, for any pseudometric $d$ on $X$, the functor  $H_n (\text{VR}(X,d),K)$  is also tame. The  symbol $\text{Tame}([0,\infty),\text{Vec}_K)$  denotes the collection of tame $[0,\infty)$-parametrized $K$ vector spaces. We refer to  the process of assigning to a  finite metric space $(X,d)$ a tame $[0,\infty)$-parametrized $K$ vector space as {\bf formal persistence} , the prominent example being the homology of the  Vietoris-Rips construction   
$ H_n (\text{VR}(X,d),K)$.

\section{The rank}\label{asgdhfyjgdm}
We  recall  how to define, calculate, and interpret the rank of a tame $[0,\infty)$-parametrized $K$ vector space. These are standard  known results,  which are included   since the rank is   the  key discrete invariant  
studied in this paper. The rank is  the number of bars in a bar decomposition. However  to calculate the rank one does  not need a bar decomposition. Calculating the rank is a much easier task than describing a bar decomposition.

\begin{subsection}{Rank of a parametrized vector space}
	Let $V$ be in $\text{Tame}([0,\infty),\text{Vec}_K)$. Choose 
	$0<t_0<\cdots < t_k$  in $[0,\infty)$ such that $V_{a\leq b}$ can fail to be an isomorphism
	only if  $a<t_i\leq b$ for some $i$. Set:
	\[H_0(V):=V_0\oplus \text{coker}(V_{0<t_0})\oplus
	\text{coker}(V_{t_0<t_1})
	\cdots\oplus \text{coker}(V_{t_{k-1}<t_k}).\]
	The vector space $H_0(V)$ is finite dimensional and does not depend on the choice of the sequence $0<t_0<\cdots < t_k$.
	Define:
	\[\text{rank}(V):=\text{dim}(H_0(V)).\]
	If   $V$ and $W$ are tame $[0,\infty)$-parametrized $K$ vector spaces, then their  direct sum
	$V\oplus W$ is also tame and
	$H_0(V\oplus W)$ is isomorphic to $H_0(V)\oplus H_0(W)$. In particular $\text{rank}(V\oplus W)=\text{rank}(V)+\text{rank}(W)$.
	Furthermore $\text{rank}(V)=0$ if and only if   $V=0$.
\end{subsection}
\begin{subsection}{Maps of parametrized vector spaces}
	A map or  a natural transformation between two  $[0,\infty)$-parametrized $K$ vector spaces $V$ and $W$,
	denoted by $f\colon V\to W$, is  a sequence
	$\{f_a\colon V_a\to W_a\}_{a\in [0,\infty)}$  of linear maps  for which the following diagram commutes for every  $a\leq b$ in $[0,\infty)$:
	\[\xymatrix@R=15pt{
		V_a\rrto^-{V_{a\leq b}}\dto_{f_a} & & V_b\dto^{f_b}\\
		W_a\rrto^-{W_{a\leq b}} & & W_b
	}\]
	
	The set of natural transformations between $V$ and $W$ is denoted by $\text{Nat}(V,W)$. Tame $[0,\infty)$-parametrized $K$ vector spaces together  with  maps between them and  the composition given by the parameter-wise composition is  a category which is denoted also by   $\text{Tame}([0,\infty),\text{Vec}_K)$. This is the only category structure we consider on the set $\text{Tame}([0,\infty),\text{Vec}_K)$. In this category   $f\colon V\to W$ is an epimorphism, a monomorphism or an isomorphism if and only if,  for every $a$,   the linear function $f_a$ is, respectively, an epimorphism, a monomorphism or an isomorphism of vector spaces.
\end{subsection}
\begin{subsection}{Bars}\label{sdfgfgshbd}
	Let $s< e$ be in $[0,\infty]$ ($s$ for start and $e$ for end). Note that $e$ might be equal to $\infty$. Define $K(s,e)$ to be the  $[0,\infty)$-parametrized $K$ vector space
	given by:
	\begin{align*}K(s,e)_a & =\begin{cases} K, & \text{if } s\leq a<e\\
			0, &\text{otherwise, } 
		\end{cases}\\
		K(s,e)_{a\leq b}\colon K(s,e)_a\to K(s,e)_b & =\begin{cases}
			\text{id}, & \text{if }\text{dim}\,K(s,e)_a= \text{dim}\,K(s,e)_b=1\\
			0,&\text{otherwise. }
		\end{cases}
	\end{align*}
	We call $K(s,e)$ the {\bf bar starting in $s$ and ending in $e$}. If $e<\infty$, then $K(s,e)$  is called {\bf finite}.
	Note that $K(s,e)$ is  tame and $\text{rank}(K(s,e))=1$. 
	
	Let $s$ be in $[0,\infty)$ and $V$ be  a $[0,\infty)$-parametrized $K$ vector space. The function $\text{Nat}(K(s,\infty), V)\to V_s$,  assigning to a map $f\colon K(s,\infty)\to V$ the element $f_s(1)$ in $V_s$, is a bijection. Thus every element $x$ in  $V_s$ yields a unique map denoted by the same symbol $x\colon K(s,\infty)\to V$ for which $x(1)=x$. Similarly a set of elements  $\{g_i\in V_{s_i}\}_{1\leq i\leq n}$  yields
	a unique map $[g_1\, \cdots\,g_n]\colon \oplus_{i=1}^n K(s_i,\infty) \to V$. Its image is denoted by $\langle g_1\, \cdots\,g_n\rangle$ and called the $[0,\infty)$-parametrized $K$ vector subspace of $V$ {\bf generated by $\{g_i\in V_{s_i}\}_{1\leq i\leq n}$}. It is the smallest subspace of $V$ containing all the $g_i$'s.  If  $\langle g_1\, \cdots\,g_n\rangle=V$, then the set $\{g_i\in V_{s_i}\}_{1\leq i\leq n}$  is  said to {\bf generate} $V$.	
	
	Assume $s< e<\infty$. Then  the function $\text{Nat}(K(s,e), V)\to V_s$, assigning to a map $f\colon K(s,e)\to V$
	the element $f_s(1)$ in $V_s$ is an inclusion. Its image coincides with $\text{ker}(V_{s\leq e}\colon V_s\to V_e)$.
	Thus every element $x$ in $\text{ker}(V_{s\leq e}\colon V_s\to V_e)$ yields a unique  $f\colon K(s,e)\to V$ for which $f_s(1)=x$.  This map is also denoted by the symbol $x$.
\end{subsection}
\begin{subsection}{Monotonicity of the rank}\label{asdgsdfhdfjn}
	Let $V$ be in $\text{Tame}([0,\infty),\text{Vec}_K)$.
	For any choice  of a finite set $\{g_i\in V_{s_i}\}_{1\leq i\leq n}$,  the subspace $\langle g_1\, \cdots\,g_n\rangle$ is tame.
	Furthermore $\text{rank}(\langle g_1\, \cdots\,g_n\rangle)\leq \text{rank}(V)$. We refer to this  property as the {\bf monotonicity of the rank}.
\end{subsection}

\begin{subsection}{Rank and the number of generators}
	Let $f\colon V\to W$ be a map between tame $[0,\infty)$-parametrized $K$ vector spaces.
	Let $0<t_0<\cdots < t_k$ be in $[0,\infty)$ such that $V_{a\leq b}$ or   $W_{a\leq b}$ can fail to be an isomorphism
	only if  $a<t_i\leq b$ for some $i$. For every $i>0$, there is a unique linear map
	$\overline{f_i}\colon \text{coker}(V_{t_{i-1}<t_i})\to \text{coker}(W_{t_{i-1}<t_i})$ making the following
	diagram commutative:
	\[\xymatrix@R=15pt{
		V_{t_{i-1}}\rrto^-{V_{t_{i-1}<t_i}} \dto_{f_{t_{i-1}}}& &V_{t_i}\rto \dto^{f_{t_{i}}} & \text{coker}(V_{t_{i-1}<t_i})
		\dto^{\overline{f_i}}\\
		W_{t_{i-1}}\rrto^-{W_{t_{i-1}<t_i}} && W_{t_i}\rto & \text{coker}(W_{t_{i-1}<t_i})
	}\]
	Define  $H_0(f)\colon H_0(V)\to H_0(W)$ to be $f_0\oplus\bigoplus_{i=1}^k \overline{f_i}$. Again the map
	$H_0(f)$ does not depend on the choice of the sequence $0<t_0<\cdots < t_k$.
	
	It turns out that $f\colon V\to W$ is an epimorphism  if and only if  $H_0(f)\colon H_0(V)\to H_0(W)$ is  surjective. This is the key observation that can be used to show that {\bf   $\text{\rm rank}(V)$ coincides with   the smallest number of elements  generating $V$}.  In particular  tame $[0,\infty)$-parametrized $K$ vector spaces are finitely generated.
\end{subsection}

\begin{subsection}{Ends of elements}\label{afsadfghsf}
	Let   $V$ be a tame  $[0,\infty)$-parametrized $K$ vector space.
	For   $x\not=0$   in $V_{s}$,    consider $L(x):=\{t\in [s,\infty)\ |\ 
	V_{s\leq t}(x)\not =0\}$ and define  {\bf the end of $x$} to be $e(x):=\text{sup}(L(x))$. 
	Note that either $e(x)=\infty$ or $e(x)<\infty$ in which case tameness implies $V_{s\leq e(x)}(x)=0$.
	The induced map $x\colon K(s,e(x))\to V$   is a monomorphism. 
\end{subsection}

\begin{subsection}{Rank and the number of bars}
	Let $V$ be a tame $[0,\infty)$-parametrized $K$ vector space.
	An element 
	$x\not=0$   in $V_{s}$  is defined to  {\bf generate a bar in $V$}, if
	$x\colon K(s,e(x))\to V$ has a retraction, i.e., a map $r\colon V\to  K(s,e(x))$ for which the following composition is the identity:
	\[\xymatrix{
		K(s,e(x))\rto^-{x}\ar@/_15pt/[rr]|{\text{id}} & V\rto^-{r} & K(s,e(x))
	}\]
	In this case $K(s,e(x))$ is a direct summand of $V$.
	
	A sequence  of elements  $\{g_i\in V_{s_i}\}_{1\leq i\leq n}$ is called a {\bf sequence of bar generators } for $V$ if the induced map $[g_1\, \cdots\,g_n]\colon \oplus_{i=1}^n K(s_i,e(g_i)) \to V$ is an isomorphism. The fundamental structure theorem states that  every  tame $[0,\infty)$-parametrized $K$ vector space admits a sequence of bar generators. In particular   every  tame $[0,\infty)$-parametrized $K$ vector space is isomorphic to a direct sum of bars. 
	
	This structure theorem can be proven by induction on the rank. If $\text{rank}(V)=0$, then the empty sequence  is a sequence of bar generators. Let $\text{rank}(V)=n>0$. Assume the statement is true if the rank is smaller than $n$.
	Let $\{x_i\in V_{s_i}\}_{1\leq i\leq n}$  be a set of generators of $V$. Set  $l:=\text{max}\{e(x_1)-s_1,\ldots,e(x_n)-s_n\}$. Among $\{x_i\ |\ e(x_i)-s_i=l\}$ choose $x_j$ 
	for which $e(x_j)$ is the largest.  We claim that $x_j$ generates  a bar. This implies that $V$ is isomorphic to $K\left(s_j,e(x_j)\right)\oplus W$. Thus  $\text{rank}(W)=n-1$ and by induction $W$ admits a sequence of bar generators. 
\end{subsection}

The discrete invariant we focus on in this paper is the rank or equivalently the minimal number of generators, or the number of bar generators:
\[\text{rank}\colon\text{Tame}([0,\infty),\text{Vec}_K)\to {\mathbf N}.\]
Our aim is to study its hierarchical stabilizations as explained in Section~\ref{sec:hier_stab}.
For that we need to   produce pseudometrics on $\text{Tame}([0,\infty),\text{Vec}_K)$.
Noise systems in \cite{MR3735858} were introduced exactly for this purpose. For implementing on a computer so called simple noise systems \cite[Definition 8.2]{OliverWojtek} are  more convenient. The reason is that simple noise systems are parametrized by {\bf  contours} \cite[Theorem 9.6]{OliverWojtek}. Instead of explaining the theory behind noise systems, we focus in this article on discussing only contours and how they can directly be used to define pseudometrics on $\text{Tame}([0,\infty),\text{Vec}_K)$.  Contours are also effective in   calculating induced   hierarchical stabilizations of the rank.  We believe however that it is important  to be aware of the relation between contours and noise systems. 

\section{Contours}\label{generalized_persistence and persistence contours}

\begin{definition} \label{perscont}
	A {\bf  contour} is a function $C: [0,\infty] \times [0,\infty) \ra [0,\infty]$  satisfying  the following inequalities   for all $a$ and  $b$ in $[0,\infty]$ and  $\epsilon$ and $\tau$  in $[0,\infty)$:
	\begin{enumerate}
		\item if $a \leq b$ and $\epsilon \leq \tau$, then $C(a,\epsilon) \leq C(b,\tau)$;
		\item $a \leq C(a,\epsilon)$;
		\item $C(C(a,\epsilon),\tau) \leq C(a,\epsilon + \tau)$.
	\end{enumerate}
	
	Let $C$ be a contour. $C$ is an  {\bf action} if   $C(a,0)= a$ and $C(C(a,\tau),\epsilon)= C(a,\tau+\epsilon)$ for all $a$ in $[0,\infty]$ and
	$\tau$ and $\epsilon$ in $[0,\infty)$.  $C$ is  {\bf closed} if the set $\{\epsilon\in [0,\infty)\  |\  C(a,\epsilon)\geq b\}$ is closed for all $a< b$  in $[0,\infty]$.  $C$  is  {\bf regular} if the following  conditions are satisfied:
	\begin{itemize}
		\item   $C(-,\epsilon )\colon[0,\infty]\to [0,\infty]$ 
		is a monomorphism  for every $\epsilon$  in $[0,\infty)$, 
		\item   $C(a,-)\colon[0,\infty)\to [0,\infty]$ is a monomorphism whose image is $[a,\infty)$
		for every $a$ in $[0,\infty)$.
	\end{itemize}
	
\end{definition}

The first condition of~\ref{perscont} makes sure that a contour  preserves the poset structures.  The second and third  one can be depicted graphically as:
\[\xymatrix{
	\bullet\ar@{|->}@/^10pt/[rr]|-{C(\bullet,\epsilon)}\ar@{|->}@/_16pt/[rrrrrr]|-{C(\bullet,\epsilon+\tau)} &\leq & \bullet \ar@{|->}@/^10pt/[rr]|-{C(\bullet,\tau)}& \leq & \bullet &\leq & \bullet
}\]

If $C$  is a regular contour, then for $a< b$ in $[0,\infty]$:
\[\{\epsilon\in [0,\infty)\ |\ C(a,\epsilon)\geq b\}=
\begin{cases} 
\{\epsilon\in [0,\infty)\ |\ \epsilon\geq C(a,-)^{-1}(b)\}, &\text{ if } a< b<\infty\\
\emptyset, &\text{ if } a<b=\infty.
\end{cases}
\]
Since all the sets on the right above are closed, regular contours are therefore  closed. For contours to be useful as tools in data analysis we need methods to produce them. We now present several of them  along with examples. 

Definition \ref{perscont} gives  three functional inequalities implicitly characterizing    contours. The last  inequality however makes it difficult to give explicit formulas. We can in any case make initial guesses for the  form of a contour and then try to find a formula satisfying the requirements  of Definition~\ref{perscont}.

\setcounter{subsection}{1}
\begin{subsection}{Exponential contours}
	Let $f\colon [0,\infty)\to  [0,\infty)$ be a  non-decreasing function such that $f(0)\geq 1$.
	For $(a,\epsilon) $ in  $[0,\infty] \times [0,\infty)$, define $C(a,\epsilon) :=f(\epsilon)a$.
	Then  $C$ satisfies the first two inequalities   of~\ref{perscont}. 
	The third inequality  is equivalent to:
	\[C(C(a,\epsilon),\tau) = C(f(\epsilon)a,\tau) = f(\tau)f(\epsilon)a \leq f(\epsilon + \tau)a = C(v,\epsilon + \tau).\]
	For instance, since $e^\tau e^\epsilon = e^{\epsilon + \tau}$, the function   $C$ associated with the exponential function $e^x$  is a  contour. In fact we could choose any positive base number $r$ other than $e$.  Such contours are called {\bf exponential}. Exponential contours are  actions.
\end{subsection}

\begin{subsection}{Standard contour}\label{standardcont}
	Let $f\colon [0,\infty)\to [0,\infty)$  be a  non-decreasing function.
	For $(a,\epsilon) $ in  $[0,\infty] \times [0,\infty)$, define $C(a,\epsilon) := a + f(\epsilon).$ Then $C$ satisfies  the first two inequalities of~\ref{perscont}. The third inequality is equivalent to   $a + f(\epsilon) + f(\tau) \leq a + f(\epsilon + \tau)$. Thus for $C$ to be a contour, $f$  should  be  superlinear:  $f(\epsilon) + f(\tau) \leq f(\epsilon + \tau)$. For  example  $C(a,\epsilon) = a + \epsilon$ is a contour  called the  {\bf standard contour}.  The standard contour is an action which is regular. Another  example is the {\bf parabolic contour} $C(a,\epsilon) = a + \epsilon^2$. The parabolic contour is not an action, however it is regular. In fact all contours of the form $C(a,\epsilon) = a + f(\epsilon)$ are regular 
	if \(f\) is superlinear and strictly increasing.
\end{subsection}
\smallskip

Contours can also be described by  integral equations. In Sections~\ref{sdfbhndfjhgh} and~\ref{shift type} we consider a Lebesgue measurable function  $f\colon [0,\infty)\to (0,\infty)$  with strictly positive values referred to as a {\bf density}. In Section \ref{contour_usage} we illustrate visualizations of some densities and  the associated contours of the following two  types.

\begin{subsection}{Contours of distance type}\label{sdfbhndfjhgh}
	Since $f$ has strictly positive values, for $(a,\epsilon) $ in  $[0,\infty) \times [0,\infty)$, there is a unique 
	$\text{D}_f(a,\epsilon)$ in $[a,\infty)$  for which: 
	$$\epsilon=\int_{a}^{D_f(a,\epsilon)}f(x) dx.$$
	Additivity of integrals gives for  $\epsilon$ and $\tau$ in $[0,\infty)$:
	\[\epsilon+\tau=\int_{a}^{D_f(a,\epsilon)}f(x) dx+\int_{D_f(a,\epsilon)}^{D_f(D_f(a,\epsilon),\tau)}f(x) dx=
	\int_{a}^{D_f(D_f(a,\epsilon),\tau)}f(x) dx\]
	implying  $D_f(D_f(a,\epsilon),\tau)=D_f(a,\epsilon+\tau)$. 
	The inequality $D_f(a,\epsilon)\leq D_f(b,\tau)$,  for $a\leq b<\infty$ and $\epsilon\leq \tau<\infty$, is a consequence of the monotonicity of integrals.  
	If in addition we set $D_f(\infty,\epsilon):=\infty$, then the obtained function $D_f\colon [0,\infty] \times [0,\infty)\to [0,\infty]$ is a contour, even an action. It is called of {\bf  distance type} as it describes 
	the distance  needed  to move from $a$ to the right  in order for the area under the graph of $f$ to reach $\epsilon$.
	Distance type contours are regular. If   density   is  the  constant  function $1$, then $\epsilon = \int_{a}^{D_1(a,\epsilon)}dx = D_1(a,\epsilon) - a$ and thus $D_1(a,\epsilon) = a+ \epsilon$ is the standard contour (see \ref{standardcont}).
\end{subsection}

\begin{subsection}{Contours of shift type}\label{shift type}
	For $a$ in $[0,\infty]$, there is a unique $y$ in $[0,\infty]$ such that $a=\int_{0}^{y}f(x)dx$. Define: 
	$$S_f(a,\epsilon):=\int_{0}^{y+\epsilon}f(x)dx.$$
	Monotonicity of integrals implies that $S_f$ satisfies the  first two inequalities of Definition~\ref{perscont}.
	Since $a=\int_{0}^{y}f(x)dx$  and $S_f(a,\epsilon)=\int_{0}^{y+\epsilon}f(x)dx$, by definition:
	\[S_f(S_f(a,\epsilon),\tau)=\int_{0}^{y+\epsilon+\tau}f(x)dx=S_f(a,\epsilon+\tau).\]
	The function $S_f$ is therefore a contour which is an action. By writing $S_f(a,\epsilon)=a+\int_{y}^{y+\epsilon}f(x)dx$ for $a=\int_{0}^{y}f(x)dx,$ we see  $S_f$ is a translation of $a$ by the $\epsilon$-step integral of the density. Therefore it is called of {\bf  shift type}. Shift type contours are regular. If the density is the constant function $1$, then $a=\int_{0}^{a}dx$ and hence $S_1(a,\epsilon)=\int_{0}^{a+\epsilon}dx=a+\epsilon$ is the standard contour.
\end{subsection}

\begin{subsection}{Truncating contours}  \label{pt trancations}

	Let  $C: [0,\infty] \times [0,\infty) \ra [0,\infty]$ be a  contour. Choose  an element  $\alpha$  in $[0,\infty]$.  For $(a,\epsilon)$ in  
	$ [0,\infty] \times [0,\infty)$ define:
	\[
	(C/\alpha)(a,\epsilon):= 
	\begin{cases}
	C(a,\epsilon), & \text{if } C(a,\epsilon)<\alpha\\
	\infty, & \text{if } C(a,\epsilon)\geq \alpha.
	\end{cases}
	\]
	For example $C/0=\infty$ and  $C/\infty=C$. We claim that $C/\alpha$ is a contour. The first two  inequalities of 	Definition~\ref{perscont} are clear.  It remains to show:
	\[(C/\alpha)\big( (C/\alpha)(a,\epsilon),\tau\big) \leq (C/\alpha)(a,\epsilon+\tau).
	\]
	The inequality is clear if $C(a,\epsilon+\tau)\geq \alpha$. Assume $C(a,\epsilon+\tau)< \alpha$. This implies that also  $C(a,\epsilon)<\alpha$ and $C(C(a,\epsilon),\tau)<\alpha$. Consequently $(C/\alpha)\big( (C/\alpha)(a,\epsilon),\tau\big) \allowbreak =C(C(a,\epsilon),\tau)$ and $(C/\alpha)(a,\epsilon+\tau)=C(a,\epsilon+\tau)$ and hence in this case  the  inequality follows from the fact that $C$ is a contour. 
	
	The contour  $C/\alpha$ is called the {\bf truncation} of $C$ at $\alpha$. If $C$ is closed, then so is its truncation $C/\alpha$ for  $\alpha$ in $[0,\infty]$. If $\alpha \leq  \beta$ in $[0,\infty]$, then for  $(a,\epsilon)$  in $[0,\infty] \times [0,\infty)$:
	\[\infty =(C/0)(a,\epsilon)\geq(C/\alpha)(a,\epsilon)\geq (C/\beta)(a,\epsilon)\geq C(a,\epsilon).
	\]
\end{subsection}
\begin{subsection}{Almost a contour}
	Let  $C: [0,\infty] \times [0,\infty) \ra [0,\infty]$ be a  contour. Choose an element  $\alpha$  in $[0,\infty]$. For $(a,\epsilon)$ in  \([0,\infty] \times [0,\infty)\) define:
	\[
	(C/\!\!/\alpha)(a,\epsilon) := \begin{cases}
	\infty, & \text{if } a=\infty\\
	C(a,\epsilon), & \text{if } a<\infty \text{ and }  C(a,\epsilon)-a<\alpha\\
	\infty, & \text{if } a<\infty \text{ and } C(a,\epsilon)-a\geq \alpha.
	\end{cases}
	\]
	Is the function $(C/\!\!/\alpha)$ a contour? The second inequality of Definition~\ref{perscont} is clear. The third inequality 
	$(C/\!\!/\alpha)\big( (C/\!\!/\alpha)(a,\epsilon),\tau\big) \leq(C/\!\!/\alpha)(a,\epsilon+\tau)$ is clear if $a=\infty$ or $a<\infty$ and $C(a,\epsilon+\tau)-a\geq \alpha$. Assume $a<\infty$ and $C(a,\epsilon+\tau)-a<\alpha$. This implies:
	\begin{align*} C(a,\epsilon)-a & \leq C(a,\epsilon+\tau)-a<\alpha,\\
		C(C(a,\epsilon),\tau)-C(a,\epsilon) & \leq C(a,\epsilon+\tau)-a<\alpha.
	\end{align*}
	Thus  $(C/\!\!/\alpha)\big( (C/\!\!/\alpha)(a,\epsilon),\tau\big)=C(C(a,\epsilon),\tau)$ and 
	$(C/\!\!/\alpha)(a,\epsilon+\tau)=C(a,\epsilon+\tau)$. The desired inequality follows    from the fact that $C$ 
	is a contour.  
	
	If $\epsilon\leq \tau$, then since $C(a,\epsilon)-a\leq C(a,\tau)-a$, the inequality $(C/\!\!/\alpha)(a,\epsilon)\leq (C/\!\!/\alpha)(a,\tau)$
	holds for any $a$.  Thus the function $C/\!\!/\alpha$ satisfies almost all of the requirements of the Definition~\ref{perscont} except possibly for the preservation 
	of the poset relation in  the first variable. This last requirement  can in fact  fail to be satisfied.  For example consider the distance  contour $D_f$ with respect to the density given in Figure~\ref{fig_point_processes_contour}. In this case $D_f(0.2,1)-0.2>1.7$ and  $D_f(0.3,1)-0.3<1.7$.  Thus
	$(D_f/\!\!/1.7)(0.2,1)=\infty$ and $(D_f/\!\!/1.7)(0.3,1)=D_f(0.3,1)<2$.  
	
	It turns out that all the results  in this article regarding contours until Theorem~\ref{agfhjnd} do not require the assumption of  the preservation of the poset relation in  the first variable.  Exploring generalizations of contours to functions that do not preserve order in the first variable   is a part of our carrent research.

\end{subsection}
Contours appeared independently  in \cite{MR3413628} under the name superlinear families
where  they were used to define interleaving distances between generalized persistence modules. Since then generalized persistence modules have gathered some interest, see \cite{MeehanMeyer} and \cite{Puuska}. In  this section  we presented  a variety of  ways of constructing contours  greatly enlarging \cite{MR3413628}, 
in which only the standard contour is given as a concrete example of  a superlinear family. The aim of the next section is to show how contours lead to  pseudometrics on $\text{Tame}([0,\infty),\text{Vec}_K)$.

\section{Constructing pseudometrics from contours}
In this section we explain how to use a  contour  to  define a   pseudometric on $\text{Tame}([0,\infty),\text{Vec}_K)$. The initial idea developed together with Oliver G\"{a}fvert and some of the text and diagrams below were written by him. It is also planned for some of this material to be  a part of Oliver's future work. We refer to the thesis work~\cite{OliverThesis} and paper~\cite{OliverWojtek} for relevant  background  that lead us together with Oliver to discover this connection between contours and pseudometrics.

To estimate   how far apart  tame parametrized vector spaces are from each other with respect  to a contour
we use the notion of equivalences:
\begin{definition}\label{sdgsfghs}
	Let $C\colon [0,\infty] \times [0,\infty)\to [0,\infty]$ be a contour,
	$V$ and $W$ be  tame $[0,\infty)$-parametrized  $K$ vector spaces, and 
	$\epsilon$  be in $[0,\infty)$. 
	\begin{enumerate}
		\item 
		A map $f\colon V\to W$ is called an {\bf $\epsilon$-equivalence}  (with respect to $C$) if, for every 
		$a$ in $[0,\infty)$ such that $C(a,\epsilon)<\infty$, there is a linear function  $W_a\to V_{C(a,\epsilon)}$ making the following diagram commutative:
		\[\xymatrix@R=15pt@C=18pt{
			V_a\dto_{f_a}\rrto^-{V_{a\leq C(a,\epsilon)}}  & & V_{C(a,\epsilon)}\dto^{f_{C(a,\epsilon)}}  \\
			W_a\rrto_-{W_{a\leq C(a,\epsilon)}}\urrto && W_{C(a,\epsilon)}
		}\] 
		\item   The objects $V$ and $W$ are called  {\bf $\epsilon$-equivalent}  (with respect to $C$) if there is a  tame $[0,\infty)$-parametrized  $K$ vector space $X$ and maps $f\colon V\to X\leftarrow W:\!g$ such that $f$ is an $\epsilon_1$-equivalence, $g$ is an $\epsilon_2$-equivalence, and $\epsilon_1+\epsilon_2\leq \epsilon$.
		\item Let $S:=\{\epsilon\in[0,\infty)\ |\ \text{$V$ and $W$ are $\epsilon$-equivalent}\}$. Define:
		\[d_C(V,W):=
		\begin{cases}
		\infty, &\text{ if } S=\emptyset\\
		\text{\rm inf}(S), &\text{ if } S\not=\emptyset.
		\end{cases}\]
	\end{enumerate}
\end{definition}

If $C=\infty$, then every map is a $0$-equivalence,  all   tame $[0,\infty)$-parametrized  $K$ vector spaces $V$ and $W$ are $0$-equivalent,  and   $d_C(V,W)=0$. A monomorphism  $f\colon V\to W$ in $\text{\rm Tame}([0,\infty),\text{\rm Vec}_K)$ is an $\epsilon$-equivalence with respect to $C$ if and only if the image of  $W_{a\leq C(a,\epsilon)}\colon W_a\to W_{C(a,\epsilon)}$ is included in the image of $f_{C(a,\epsilon)}\colon V_{C(a,\epsilon)}\to W_{C(a,\epsilon)}$ for all $a$ in $[0,\infty)$ such that $C(a,\epsilon)<\infty$. In particular  $0\to W$ is an  $\epsilon$-equivalence   if and only if $W_{a\leq C(a,\epsilon)}\colon W_a\to W_{C(a,\epsilon)}$  is the zero function for  all such $a$. Furthermore $W$ is $\epsilon$-equivalent to $0$ if and only if $0\to W$ is an $\epsilon$-equivalence. Thus $d_C(0,W)< \epsilon$ if and only if  $W_{a\leq C(a,\epsilon)}\colon W_a\to W_{C(a,\epsilon)}$  is the zero function for all $a$ in $[0,\infty)$ such that $C(a,\epsilon)<\infty$. It  is however not true in general that if $\epsilon=d_C(0,W)<\infty$, then $W_{a\leq C(a,\epsilon)}\colon W_a\to W_{C(a,\epsilon)}$  is the zero  function for  $a$ in $[0,\infty)$ such that $C(a,\epsilon)<\infty$. This depends if $C$ is closed or not.

According to Definition~\ref{sdgsfghs}, to estimate and  calculate $d_C(V,W)$, the objects $V$ and $W$ are compared through a  third tame $[0,\infty)$-parametrized vector space via a short zig-zag of equivalences $V\to X\leftarrow W$.  In principal  to assure the triangular inequality and obtain a pseudometric one should  compare $V$ and $W$   not via  short but long zig-zags $V\to X_0\leftarrow \cdots \to X_k\to W$ of equivalences. The main content of the next proposition is that in our case short zig-zags are sufficient.    

\begin{proposition}\label{adsfgdsfhgbd}
	$d_C$ is a pseudometric on $\text{\rm Tame}([0,\infty),\text{\rm Vec}_K)$.
\end{proposition}

To prove this proposition and explain why short zig-zags are sufficient we  need:

\begin{proposition}\label{prop basicequiv}
	Let $C\colon [0,\infty] \times [0,\infty)\to [0,\infty]$ be a contour and $U$, $V$, and $W$ 
	be  tame $[0,\infty)$-parametrized  $K$ vector spaces.
	\begin{enumerate}
		\item  Composition of $\tau$- and $\epsilon$-equivalences is an $(\tau+\epsilon)$-equivalence.		
		\item  In the following   pushout square, $P$ is also tame and, if  $f$ is an $\epsilon$-equivalence, then so is $g$:
		\[\xymatrix@R=13pt@C=15pt{
			V\dto_-f \rto & U\dto^-{g}\\
			W\rto & P
		}\] 
	\end{enumerate}
\end{proposition}
\begin{proof}
	\noindent
	(1):\quad 
	Consider an $\epsilon$-equivalence $g\colon U\to V$ and a $\tau$-equivalence $f\colon V\to W$.
	If $C(a,\tau+\epsilon)<\infty$, then  $C(a,\tau)\leq C(C(a,\tau),\epsilon)\leq C(a,\tau+\epsilon)<\infty$ and hence, for any such $a$, there are linear functions  $W_a\to V_{C(a,\tau)}\to U_{C(C(a,\tau),\epsilon)}$ making the following diagram commutative: 
	\[\xymatrix@R=10pt@C=25pt{
		U_a\dto_{g_a}\rto & U_{C(a,\tau)}\dto\rto & U_{C(C(a,\tau),\epsilon)}\dto \rto & U_{C(a,\tau+\epsilon)}\dto\\
		V_a\dto_{f_a}\rto & V_{C(a,\tau)}\dto\rto\urto & V_{C(C(a,\tau),\epsilon)}\dto \rto & V_{C(a,\tau+\epsilon)}\dto\\
		W_a\rto \urto& W_{C(a,\tau)}\rto & W_{C(C(a,\tau),\epsilon)} \rto & W_{C(a,\tau+\epsilon)}
	}\]
	The  diagonal morphism $W_a\to U_{C(a,\tau+\epsilon)}$ in this diagram  is a linear function whose existence is required for $fg$ to be an  $(\tau+\epsilon)$-equivalence.
	\smallskip
	
	\noindent
	(2):\quad  Tameness of $P$ is clear.
	Assume $C(a,\epsilon)<\infty$. Let $W_a\to V_{C(a,\epsilon)}$ be a function given by the fact that
	$f\colon V\to W$   is an $\epsilon$-equivalence.   This function fits into the following  cube where the dotted arrow  is the unique function making this cube commutative (its existence is guaranteed by the universal property of push-outs):
	\[\xymatrix@R=12pt@C=15pt{
		V_a\ddto_-{f_a}\drto\rrto & & U_a \ddto_(0.3){g_a}|\hole\drto \\
		& V_{C(a,\epsilon)}\ddto|(0.3){f_{C(a,\epsilon)}}\rrto&& U_{C(a,\epsilon)} \ddto^{g_{C(a,\epsilon)}}\\
		W_a\drto\urto\rrto|\hole  && P_a \drto \ar@{.>}[ur] \\
		&   W_{C(a,\epsilon)} \rrto & & P_{C(a,\epsilon)}
	}\]
\end{proof}

\begin{proof}[Proof of Proposition~\ref{adsfgdsfhgbd}]
	Symmetry is clear. For the triangle inequality consider $\epsilon$-equivalent   $U$ and  $V$,  and $\tau$-equivalent $V$ and  $W$ and form 
	the following diagram:
	\[\xymatrix@R=9pt@C=9pt{
		U\drto^-f & & V\dlto_-g\drto^-h & & W\dlto_-k\\
		& X\drto_-{h'} & &Y\dlto^-{g'}\\
		& & P
	}\]
	where $f$ is an $\epsilon_1$-equivalence,  $g$ is an $\epsilon_2$-equivalence, $\epsilon_1+\epsilon_2\leq\epsilon$,  $h$ is a $\tau_1$-equivalence, $k$ is a $\tau_2$-equivalence, $\tau_1+\tau_2\leq\tau$, and the central  square in this diagram is a push-out. According to~\ref{prop basicequiv}.(2),  $g'$ is  an $\epsilon_2$-equivalence and $h'$ is a $\tau_1$-equivalence. Thus~\ref{prop basicequiv}.(1) implies  $h'f$ is a $(\tau_1+\epsilon_1)$-equivalence and  $g'k$ is a $(\tau_2+\epsilon_2)$-equivalence. Since $\epsilon+\tau\geq \tau_1+\epsilon_1+\tau_2+\epsilon_2$, we can  conclude that  $U$ and $W$ are $(\epsilon+\tau)$-equivalent. The triangle inequality $d_C(U,W)\leq d_C(U,V)+d_C(V,W)$ follows.
\end{proof}

We are interested in  not just  individual pseudometrics but also in their sequences, particularly the non-decreasing ones (see Definition~\ref{asfgdsfghs}). To produce such sequences of pseudometrics on $\text{Tame}([0,\infty],\text{Vec}_K)$ we use:

\begin{proposition}\label{asdfgdsfhgfdn}
	Assume $C$ and $D$ are  contours such that   $C\geq D$. Then:
	\begin{enumerate}
		\item  An $\epsilon$-equivalence with respect to $D$ implies $\epsilon$-equivalence with respect to $C$.
		\item $d_C(V,W)\leq d_D(V,W)$.
	\end{enumerate}
\end{proposition}
\begin{proof}
	Statement (2) is a direct consequence of (1).  To show  (1), 
	let $f\colon V\to W$ be an $\epsilon$-equivalence with respect to $D$.
	If  $C(a,\epsilon)<\infty$, then also $D(a,\epsilon)<\infty$,
	and hence we can form the following commutative diagram whose  diagonal  is the function assuring  $f$ is an $\epsilon$-equivalence with respect to $C$:
	\[\xymatrix@R=13pt@C=18pt{
		V_a\dto_{f_a}\rrto^-{V_{a\leq D(a,\epsilon)}}  & & V_{D(a,\epsilon)}\dto^{f_{D(a,\epsilon)}}\ar@{->}[rrr]^-{V_{D(a,\epsilon)\leq C(a,\epsilon)}} &&& V_{C(a,\epsilon)}\dto^{f_{C(a,\epsilon)}} \\
		W_a\rrto_-{W_{a\leq D(a,\epsilon)}}\urrto && W_{D(a,\epsilon)}\ar@{->}[rrr]_-{W_{D(a,\epsilon)\leq C(a,\epsilon)}}  && & W_{C(a,\epsilon)}
	}\] 
\end{proof}

\section{Stable ranks}
We are now ready to discuss  our models for supervised persistence:
\begin{definition}
	Let $C$ be a contour and $d_C$ be the associated pseudometric on $\text{\rm Tame}([0,\infty),\text{\rm Vec}_K)$
	(see Proposition~\ref{adsfgdsfhgbd}). 
	The hierarchical stabilization (see Section~\ref{sec:hier_stab}) of  the rank function   $\text{\rm rank}\colon \text{\rm Tame}([0,\infty),\text{\rm Vec}_K)\to {\mathbf N}$  (see Section~\ref{asgdhfyjgdm}), with respect to $d_C$, is called {\bf  stable rank} and is denoted by:
	\[\widehat{\text{\rm rank}}_{C}\colon  \text{\rm Tame}([0,\infty),\text{\rm Vec}_K)\to {\mathcal M}.\]
\end{definition}
By Definition~\ref{asdghsfghdfg}, the stable rank assigns to  a tame $[0,\infty)$-parametrized  $K$ vector space $V$, the function $\widehat{\text{rank}}_{C} V\colon [0,\infty)\to[0,\infty)$   defined as follows:
\[\widehat{\text{rank}}_{C} V(t)=\text{min}\left\{\text{rank}(W)\   |\   W\in  \text{Tame}([0,\infty),\text{Vec}_K)\text{ and }d_C(V,W)\leq t\right\}.\]
Thus $\widehat{\text{rank}}_{C} V$ is non-increasing   with natural numbers as values and therefore  there are finitely many elements $0<\tau_0<\cdots< \tau_n$ in its domain $[0,\infty)$ such that $\widehat{\text{rank}}_{C} V$ is constant on the open intervals $(0,\tau_0)$,\ldots, $(\tau_{i},\tau_{i+1})$,\ldots, $(\tau_n,\infty)$. Depending on the contour,   $\widehat{\text{rank}}_{C} V$ may fail  to  be right or left continuous.

The aim of this section is to provide  effective ways of calculating the stable rank.
If a contour is closed (see Definition~\ref{perscont}), then the following fundamental properties of  the stable rank  explain how its values are related to a bar decomposition. One can then use for example the Ripser software~\cite{Ripser} for effective calculations of the stable rank.

\begin{theorem}\label{agfhjnd}
	If $C$ is  a closed contour, then 
	$\widehat{\text{\rm rank}}_{C}\colon  \text{\rm Tame}([0,\infty),\text{\rm Vec}_K)\to {\mathcal M}$
	satisfies  the following properties:
	\begin{enumerate}
		\item  The function $\widehat{\text{\rm rank}}_{C}(V)\colon [0,\infty)\to[0,\infty)$ is  right continuous for any $V.$
		\item  The function  $\widehat{\text{\rm rank}}_{C}$ is linear: $\widehat{\text{\rm rank}}_{C}(V\oplus W)=\widehat{\text{\rm rank}}_{C}(V)+\widehat{\text{\rm rank}}_{C}(W)$.
		\item $ \widehat{\text{\rm rank}}_{C}\left(\bigoplus_{i=1}^n K(s_i,e_i)\right)(t)=|\{i\ |\  C(s_i,t)<e_i\}|$.
	\end{enumerate}
\end{theorem}

According to the third statement of Theorem~\ref{agfhjnd}, the values of the stable rank are certain  counts of bars 
in a bar decomposition. Note that in our  entire set up of the hierarchical stabilization process and in the 
definition of the stable rank  no bar decomposition is  mentioned or used. The aim of the hierarchical stabilization
is to convert discrete invariants into  stable invariants by minimizing over discs.   The stable rank therefore encodes  in a stable  way some information about how ranks of tame $[0,\infty)$-parametrized vector spaces change in certain neighbourhoods. Achieving stability is the key objective of this process. The linearity property (Theorem~\ref{agfhjnd}.(2)) is then what connects the stable rank with  bar  decompositions. Traditionally, in persistent homology one considers  bar decompositions first and then proves that with respect to certain metrics  the associated persistence diagrams are stable. The reversal  of this  perspective, {\bf stability first and 	decompositions after}, has been an important step in our approach to  homological persistence. Stability is so fundamental for any method aimed at data analysis that we believe it should be the primary guiding principle in
homological  persistence.  Decompositions can be then used as effective tools for calculating the constructed stable invariants. This change of perspective is vital to multiparameter generalizations of homological persistence where decomposition methods are not available (see~\cite{OliverWojtek}).

Theorem~\ref{agfhjnd} is a direct consequence of Corollary~\ref{asfdfghfdghjn} and Proposition~\ref{asfgdfsghbn}.
This  strategy to prove Theorem~\ref{agfhjnd}   is taken from~\cite{OliverWojtek}
(see  \cite[Section  8]{OliverWojtek}) and  is based on:  
\begin{definition}
	Let $C$ be a contour,  $t$  in $[0,\infty)$, and $V$ in $\text{\rm Tame}([0,\infty),\text{\rm Vec}_K)$.
	The {\bf $t$-shift} of $V$ with respect to $C$, denoted by $V_C[t]$,  is the $[0,\infty)$-parametrized  $K$ vector subspace of $V$  generated by all the elements in the images of the transition functions 
	$V_{a\leq C(a,t)}\colon V_a\to V_{C(a,t)}$ for all $a$ in $[0,\infty)$ such that $ C(a,t)<\infty$.
\end{definition}

The   shift operation enjoys the following properties:

\begin{proposition}\label{asfgdfsghbn}
	Let $V,W$  be  in $\text{\rm Tame}([0,\infty),\text{\rm Vec}_K)$, 
	$t$  in $[0,\infty)$, and $C$ be  a contour.
	\begin{enumerate}
		\item If $V$ is generated by $\{g_i\in V_{s_i}\}_{1\leq i\leq n}$, then $V_C[t]$ is generated by:
		\[\{V_{s_i\leq C(s_i,t)}(g_i)\ |\ 1\leq i\leq n\text{ and } C(s_i,t)<\infty\}.\]
		\item  $V_C[t]$ is tame.
		\item The inclusion $V_C[t]\subset V$ is a $t$-equivalence with respect to $C$.
		\item A monomorphism $f\colon W\subset V$   is a $t$-equivalence if and only if
		$V_C[t]$ is contained in the image of $f$.
		\item The shift is linear: $(V\oplus W)_C[t]$ and $V_C[t]\oplus W_C[t]$ are isomorphic.
		\item $\left(\bigoplus_{i=1}^n K(s_i,e_i)\right)_C[t]$ is isomorphic to $ \bigoplus_{\{i\ |\ C(s_i,t)<e_i\}}K(C(s_i,t),e_i)$.
		\item   $\text{\rm rank}\left((\bigoplus_{i=1}^n K(s_i,e_i))_C[t]\right)=|\{i\ |\ C(s_i,t)<e_i\}|$.		
	\end{enumerate}
\end{proposition}
\begin{proof}
	(1) is a direct consequence of the  definitions; (2), (3), (4), (5), and (6)  follow
	from (1), and (7) from (6).
\end{proof}

If   $t\leq t'$ in $[0,\infty)$, then $V_C[t]\supset V_C[t']$ and  therefore  by the monotonicity of rank
(see~\ref{asdgsdfhdfjn}), $\text{rank}(V_C[t])\geq\text{rank}(V_C[t'])$. Thus the function $t\mapsto \text{rank}(V_C[t])$ is non-increasing and,  similarly to the stable rank, there are finitely many elements $0<\tau_0<\cdots< \tau_n$ in  $[0,\infty)$ such that $\text{rank}(V_C[-])$ is constant on the open intervals $(0,\tau_0)$,\ldots, $(\tau_{i},\tau_{i+1})$,\ldots, $(\tau_n,\infty)$.

\begin{proposition}\label{sdGSGFJHDGHJD} Let $V$ be a tame $[0,\infty)$-parametrized  $K$ vector space. If $C$ is a closed contour, then $ \text{\rm rank}(V_C[-])$ is a right continuous function, i.e., there are finitely many  $0<\tau_0<\cdots <\tau_n$ in $[0,\infty)$  such that  $\text{\rm rank}(V_C[-])$	is constant on the  left closed and right open  intervals $[0,\tau_0),\ldots, [\tau_i,\tau_{i+1}),\ldots, [\tau_{n},\infty)$.
\end{proposition}
\begin{proof}
	It is enough to show that  for any $t$, there is $t<t'$ for which  $\text{\rm rank}(V_C[t])=\text{\rm rank}(V_C[t'])$. Let $0<t_0<\cdots<t_k$ be  such that  $V_{a\leq b}\colon V_{a}\to V_{b}$  may fail  to be an isomorphism only if  $a<t_i\leq b$ for some $i$. Choose a sequence  $\{g_i\in V_{s_i}\}_{1\leq i\leq n}$ of generators of $V$. Consider  only these $C(s_i,t)$ which lie in $[0,\infty)$. Since $C$ is closed, there are $t<t'$ for which both $C(s_i,t)$ and $C(s_i,t')$ are in one of the intervals $[0,t_0)$,\ldots, $[t_n,\infty)$. Consequently the transition functions $V_{C(s_i,t)\leq C(s_i,t')}$, for all $i$,  are  isomorphisms and  $V_C[t]$
	and $V_C[t']$ have the same rank.
\end{proof}

Here is the  key relation between the stable rank and the shift operation:
\begin{theorem}\label{asfsafdshj}
	Let  $C$ be a contour and  $V$ be in $\text{\rm Tame}([0,\infty),\text{\rm Vec}_K)$. Then for $t<t'$ in $[0,\infty)$:
	\[ \text{\rm rank}(V_C[t])\geq \widehat{\text{\rm rank}}_{C}(V)(t)\geq  \text{\rm rank}(V_C[t']).\]
\end{theorem}
\begin{proof}
	Since $V_C[t]\subset V$ is a $t$-equivalence (see Proposition~\ref{asfgdfsghbn}.(3)), $d_C(V,V_C[t])\leq t$. This  gives the first inequality.
	
	Let $W$ be a tame $[0,\infty)$-parametrized  $K$ vector space for which
	$d_C(V,W)\leq t $ and $\text{rank}(W)=\widehat{\text{\rm rank}}_{C}V(t)$.
	Since $d_C(V,W)<t' $, by definition there are  maps 
	$f\colon V\to X\leftarrow W:\!g$
	in  $\text{Tame}([0,\infty),\text{Vec}_K)$
	where $f$ is $\tau_1$-equivalence, $g$ is $\tau_2$-equivalence, and $\tau_1+\tau_2<t'$.
	For any  $a$ in $[0,\infty)$ such that $C(a,\tau_1)<\infty$, consider the following commutative diagram
	where the vertical arrows are the transition functions  and $\beta_a$  is  the  lift given by the fact that $f$ is $\tau_1$-equivalence:
	\[\xymatrix@R=13pt@C=18pt{
		V_a\rto^{f}\dto &X_a\dto\dlto|-{\beta_a} & W_a\dto\lto_{g} \\
		V_{C(a,\tau_1)}\rto_{f} & X_{C(a,\tau_1)} & W_{C(a,\tau_1)}\lto^{g}
	}\]
	Let $n=\text{rank}(W)$ and $\{w_i\in W_{s_i}\}_{1\leq i\leq n}$ be a minimal set of generators of $W$.
	Set $x_i:=g(w_i)$. For any $i$ in the set $I:=\{i\ |\  C(s_i,\tau_1)<\infty\}$, define $v_i:=\beta_{s_i}(g(w_i))\in V_{C(s_i,\tau_1)}$ and  $V':=\langle v_i\ |\ i\in I\rangle\subset V$. 
	We claim that the following  inclusions hold:  \[V_C[t']\subset V_C[\tau_1+\tau_2]\subset V'\subset V,\]
	which imply $\widehat{\text{\rm rank}}_{C}V(t)=n\geq  \text{rank}(V')\geq \text{rank}(V_C[t'])$, proving the second inequality.
	The first inclusion is a consequence of  $\tau_1+\tau_2<t'$.  The last inclusion is by definition. It remains to show the middle inclusion  $V_C[\tau_1+\tau_2]\subset V'$.
	
	For any $a$ in $[0,\infty)$ such that
	$C(a,\tau_1+\tau_2)<\infty$  we have  the following commutative diagram where all the horizontal arrows indicate the transition functions, vertical arrows are functions induced by $f$ and $g$, and $\beta_{C(a,\tau_2)}$ and $\alpha_a$ are  lifts  guaranteed by the fact that $f$ is $\tau_1$-equivalence and  $g$ is $\tau_2$-equivalence, respectively:
	\[\xymatrix@C=30pt@R=15pt{
		V_a\dto _{f_a} \rto & V_{C(a,\tau_2)}\dto_{f}\rrto
		&&  V_{C(C(a,\tau_2),\tau_1)}\rto\dto^{f} & V_{C(a,\tau_1+\tau_2)}\dto^{f}\\
		X_a\rto\drto|{\alpha_a} & X_{C(a,\tau_2)}\rrto
		\urrto|{\beta_{C(a,\tau_2)}} 
		& & X_{C(C(a,\tau_2),\tau_1)}\rto & X_{C(a,\tau_1+\tau_2)}
		\\
		W_a\uto^{g_a}\rto& W_{C(a,\tau_2)}\uto_{g}\rrto &  &W_{C(C(a,\tau_2),\tau_1)}\rto\uto_{g} & W_{C(a,\tau_1+\tau_2)}\uto_{g}
	}\]
	Commutativity of this diagram implies that, for every such $a$, the image of  the transition function  $V_{a\leq C(a,\tau_1+\tau_2)}$
	belongs to  $V'$. Since $V_C[\tau_1+\tau_2]$ is generated by  these images, the inclusion $V_C[\tau_1+\tau_2]\subset V'$ holds.
\end{proof}

According to  Theorem~\ref{asfsafdshj}, the functions  $\widehat{\text{\rm rank}}_{C}(V)$ and $\text{\rm rank}(V_C[-])$ agree for all but finitely many points:
\begin{corollary}
	Let $C$ be a contour and $V$ be in   $\text{\rm Tame}([0,\infty),\text{\rm Vec}_K)$. Then
	there are   $0<\tau_0<\cdots< \tau_n$ in  $[0,\infty)$
	such that the functions $\widehat{\text{\rm rank}}_{C}(V)$ and $\text{\rm rank}(V_C[-])$ agree on 
	the open intervals $(0,\tau_0)$,\ldots, $(\tau_{i},\tau_{i+1})$,\ldots,  $(\tau_n,\infty)$. 
	In particular, for  $p\geq 1$:
	\[d_{\bowtie}\left(\widehat{\text{\rm rank}}_{C}(V), \text{\rm rank}(V_C[-])\right)=0=L_p\left(\widehat{\text{\rm rank}}_{C}(V), \text{\rm rank}(V_C[-])\right).\]
\end{corollary}

Theorem~\ref{asfsafdshj} together with Proposition~\ref{sdGSGFJHDGHJD} gives:
\begin{corollary}\label{asfdfghfdghjn}
	Assume $C$ is a closed contour.  Then $\widehat{\text{\rm rank}}_{C}(V)= \text{\rm rank}(V_C[-])$ for any $V$  in   $\text{\rm Tame}([0,\infty),\text{\rm Vec}_K)$.
\end{corollary}

We finish this section with:
\begin{corollary}\label{dsfadfssfhf}
	Assume  $C$ is a closed contour such that 
	$C(a,0)=a$ for any $a$ in $[0,\infty]$.
	Let  $V$ be a tame  $[0,\infty)$-parametrized  $K$ vector space. Then:
	\begin{enumerate}
		\item $\widehat{\text{\rm rank}}_{C}V(0)=\text{\rm rank}(V)$.
		\item $\widehat{\text{\rm rank}}_{C}V=0$ if and only if  $V=0$.
	\end{enumerate}
\end{corollary}
\begin {proof}
Corollary~\ref{asfdfghfdghjn}    and Proposition~\ref{asfgdfsghbn}.(1) imply 
$\widehat{\text{\rm rank}}_{C}V(0)=\text{rank}(V_{C}[0])=\text{rank}(V)$.
Since 
$\widehat{\text{\rm rank}}_{C}V$ is  non-increasing, the identity  $\widehat{\text{\rm rank}}_{C}V=0$
is equivalent to  $\widehat{\text{\rm rank}}_{C}V(0) \allowbreak =0$, which 
by  statement (1) is equivalent to $\text{rank}(V)=0$,  proving (2).
\end{proof}
\section{Life span}\label{adfsdfhgfhn}
Let 
$s<e$ be  in $[0,\infty]$.  Since $\text{rank}(K(s,e))=1$, then, for a contour $C$, the value of 
$\widehat{\text{rank}}_{C} K(s,e)(t)$ is either $1$ or $0$.
As the function $\widehat{\text{rank}}_{C} K(s,e) $ is non-increasing,
there is $l$ in $[0,\infty]$ such that:
\[\widehat{\text{rank}}_{C}K(s,e)(t)=\begin{cases} 1, & \text{if } t<l\\
0, & \text{if } t> l.
\end{cases}\]
We define  $\text{life}_C K(s,e):=l$   and  call this element in $[0,\infty]$ the {\bf life span} of  $K(s,e)$. If $l=\text{life}_C K(s,e) <\infty$, then the value $\widehat{\text{rank}}_{C}K(s,e)(l)$
can be either $1$ or $0$, depending  on the contour. 
If  $C$ is closed, then according to  Theorem~\ref{agfhjnd}.(1),
$\widehat{\text{rank}}_{C}K(s,e)(l)=0$. This with the additivity property in Theorem~\ref{agfhjnd}.(3)
gives:
\begin{proposition}\label{asfdsdfgfd}
If  $C$ is  a closed contour, then:
\[  \widehat{\text{\rm rank}}_{C}\left(\bigoplus_{i=1}^n K(s_i,e_i)\right)(t) =|\{i\ |\ t<\text{\rm life}_CK(s_i,e_i)\}|.\]
\end{proposition}

According to Proposition~\ref{asfdsdfgfd}, the  stable rank, with respect to a closed contour,    counts  bars whose life span strictly exceeds $t$.  Here is how to calculate the life span for regular contours:
\begin{proposition}\label{sasfagdfhgb}
Let $s<e$ be in $[0,\infty]$ and $\alpha$ be in $[0,\infty)$.
If  $C$ is a regular contour (see Definition~\ref{perscont}), then:
\begin{align*}
	\text{\rm life}_C K(s,e) &= \begin{cases}
		\infty, &\text{if } e=\infty\\
		C(s,-)^{-1}(e),&\text{if } e<\infty,
	\end{cases} \\
	\text{\rm life}_{C/\alpha} K(s,e) &= 
	\begin{cases}  0, &\text{if }   \alpha\leq s\\
		C(s,-)^{-1}(\alpha), & \text{if }  s<\alpha\leq e\\
		C(s,-)^{-1}(e), &\text{if }  e<\alpha.
	\end{cases}
\end{align*}
\end{proposition}

\begin{proof}
According to Theorem~\ref{agfhjnd}.(3):
\[\widehat{\text{\rm rank}}_{C}K(s,e)(t)=\begin{cases}
1, &\text{ if } C(s,t)<e\\
0,&\text{ if } C(s,t)\geq e.
\end{cases}\]
This together with the regularity of $C$ implies all the  claimed equalities.
\end{proof}
\begin{corollary}\label{afgdjhdghj}
If  $C$ is  a regular contour, then:
\begin{align*}  
	\text{\rm lim}\left(   \widehat{\text{\rm rank}}_{C}\left(\bigoplus_{i=1}^n K(s_i,e_i)\right) \right) &=
	|\{i\ |\ e_i=\infty\}|.
\end{align*}

\end{corollary}
\section{Regular contours and ampleness}
Let $C$ be a contour. For every $\alpha$ in $[0,\infty]$, we can  take its truncation $C/\alpha$ (see~\ref{pt trancations}).  In this way we get a sequence of contours indexed by $[0,\infty]$ such that for  $\alpha < \beta$ in $[0,\infty]$:
\[\infty =C/0\geq\cdots \geq C/\alpha\geq\cdots\geq C/\beta\geq\cdots\geq C/\infty=C.\]
Each of these contours  induces a pseudometric on $\text{Tame}([0,\infty),\text{Vec}_K)$ as defined in~\ref{sdgsfghs}. In this way, according to Proposition~\ref{adsfgdsfhgbd}, we obtain a sequence of pseudometrics
$\{d_{C/\alpha}\}_{\alpha\in [0,\infty]}$. Furthermore, for all $\alpha<\beta$ in $[0,\infty]$ and  $V,W$ in  $\text{\rm Tame}([0,\infty),\text{\rm Vec}_K)$, Proposition~\ref{asdfgdsfhgfdn}.(2) gives the following  inequalities: 
\[
0=d_{C/0}(V,W)\leq \cdots \leq d_{C/\alpha}(V,W)\leq\cdots\leq  d_{C/\beta}(V,W)\leq\cdots\leq d_{C}(V,W).
\]
A  contour therefore induces  a non-decreasing  sequence of pseudometrics on  $\text{Tame}([0,\infty),\text{Vec}_K)$, leading to hierarchical stabilization (Definition~\ref{asfgdsfghs}.(4)):
\[\xymatrix@R=35pt{
\text{Contours}\ar[d]_-{\{d_{-/\alpha}\}_{\alpha\in [0,\infty]}}\\
\text{Sequences of  pseudometrics on $\text{Tame}([0,\infty),\text{Vec}_K)$}
}
\ \ 
\xymatrix@R=35pt{
\text{Tame}([0,\infty),\text{Vec}_K)\ar[d]_{\overline{\text{rank}}_C} \\ {\mathcal M}_2 
}
\]

We are now ready to state and prove our key ampleness  result (see Definition~\ref{sdfgdghdgjn} and discussion after):

\begin{theorem}\label{adfgsdfg} Consider the category  $\text{\rm Tame}([0,\infty),\text{\rm Vec}_K)$.
If $C$ is a regular contour, then the sequence of pseudometrics $\{d_{C/\alpha}\}_{\alpha\in [0,\infty]}$ is ample for the rank  function $\text{\rm rank}\colon \text{\rm Tame}([0,\infty),\text{\rm Vec}_K)\to {\mathbf N}$.
\end{theorem}

\begin{proof}
Let $V$ and $W$ be tame  $[0,\infty)$-parametrized  $K$ vector spaces. Assume, for every $\alpha$ in $[0,\infty]$, $\widehat{\text{\rm rank}}_{C/\alpha}V=\widehat{\text{\rm rank}}_{C/\alpha}W$.   
We need to show  $V$ and $W$ are isomorphic. 

Since $C$ is regular, then it is closed and  $C(a,0)=a$ for all $a$, and hence according to Corollaries~\ref{asfdfghfdghjn} and \ref{dsfadfssfhf}:
\[\text{rank}(V)=\text{rank}(V_C[0])=\widehat{\text{\rm rank}}_{C}V(0)=
\widehat{\text{\rm rank}}_{C}W(0)=\text{rank}(W_C[0])=\text{rank}(W).\]
Thus $V$ and $W$  have the same rank. Assume $V$ is isomorphic to $\oplus_{i=1}^nK(s_i,e_i)$ and $W$ is isomorphic to $\oplus_{i=1}^nK(s'_i,e'_i)$. 
\smallskip

\noindent
{\bf Step 1: Reduction to finite bars.}
According to Corollary~\ref{afgdjhdghj}:
\[|\{i \ |\ e_i=\infty\}|=\text{lim}\left(\widehat{\text{\rm rank}}_{C}V\right)=\text{lim}\left(\widehat{\text{\rm rank}}_{C}W\right)=|\{i \ |\ e'_i=\infty\}|.\]
Thus $V$ and $W$ are isomorphic to, respectively:
\[\bigoplus_{i=1}^{n_1}K(s_i,e_i)\oplus \bigoplus_{j=1}^{n_2}K(s_j,\infty)\ \ \ \ \ \ \ \ \ \ 
\bigoplus_{i=1}^{n_1}K(s'_i,e'_i)\oplus \bigoplus_{j=1}^{n_2}K(s'_j,\infty),\]
where  $e_i,e'_i<\infty $ for  $i=1,\ldots,n_1$. 

Choose $\beta$ in $[0,\infty)$   such that $\beta> e_i, e'_i, s_j, s'_j$ for all $i$ and $j$, and define:
\[V/\beta:=\bigoplus_{i=1}^{n_1}K(s_i,e_i)\oplus \bigoplus_{j=1}^{n_2}K(s_j,\beta)\ \ \ \  \ \ \ \ 
W/\beta:=\bigoplus_{i=1}^{n_1}K(s'_i,e'_i)\oplus \bigoplus_{j=1}^{n_2}K(s'_j,\beta).\]
Note that $V$ and $W$ are isomorphic if and only if $V/\beta$ and $W/\beta$ are isomorphic. Thus to prove the theorem it is enough to show   $V/\beta$ and $W/\beta$ are isomorphic.

We claim that, for every $\alpha$ in $[0,\infty]$, 
$\widehat{\text{\rm rank}}_{C/\alpha}(V/\beta)=\widehat{\text{\rm rank}}_{C/\alpha}(W/\beta)$.
This follows from the assumption  $\widehat{\text{\rm rank}}_{C/\alpha}V=\widehat{\text{\rm rank}}_{C/\alpha}W$,  
the additivity of the stable rank (Theorem~\ref{agfhjnd}.(2)),  and Proposition~\ref{sasfagdfhgb} which gives that for all $s<\beta$:
\[\begin{array}{ccc}
\text{\rm life}_{C/\alpha}K(s,\beta)=C^{-1}(s,\alpha)=\text{\rm life}_{C/\alpha}K(s,\infty),
&\text{if } \alpha\leq \beta,\\
\text{\rm life}_{C/\alpha}K(s,\beta)=C^{-1}(s,\beta)=\text{\rm life}_{C/\beta}K(s,\infty),
&\text{if } \alpha> \beta.
\end{array}
\]

We reduced the theorem to the case when all the bars in the bar decompositions of $V$ and $W$ are finite. 
\smallskip

\noindent
{\bf Step 2: Induction on the rank.}
Assume 
$V$ is  isomorphic to $\oplus_{i=1}^nK(s_i,e_i)$ and 
$W$ is    isomorphic to $\oplus_{i=1}^nK(s'_i,e'_i)$ where $e_i,e'_i<\infty$. We are going to prove by induction 
on the rank  that $V$ and $W$ are isomorphic.  The statement is clear if  $\text{rank}(V)=\text{rank}(W)=0$, since in this case both $V$ and $W$ are isomorphic to  $0$.
Assume $n=\text{rank}(V)>0$.
Let $l_i=\text{life}_C K(s_i,e_i) $ and $l_i'=\text{life}_C K(s'_i,e'_i) $ (see Section~\ref{adfsdfhgfhn}). 
Recall that according to Proposition~\ref{asfdsdfgfd}, for  $t$ in $[0,\infty)$:
\[|\{i\ |\ t < l_i\}|=\widehat{\text{\rm rank}}_{C}V(t)=\widehat{\text{\rm rank}}_{C}W(t)=|\{i\ |\ t < l'_i\}|.\]
It follows that $l_{\text{max}}:=\text{max}\{l_i\ |\ 1\leq i\leq n\}=\text{max}\{l'_i\ |\ 1\leq i\leq n\}$.
Let $e_{\text{max}}:=\text{max}\{e_i\ |\ l_i=l_{\text{max}}\}$ and $e'_{\text{max}}:=\text{max}\{e'_i\ |\ l'_i=l_{\text{max}}\}$. We claim that $e_{\text{max}}=e'_{\text{max}}$. If $e_{\text{max}}<e'_{\text{max}}$, then
$\widehat{\text{\rm rank}}_{C/e_{\text{max}}}V >\widehat{\text{\rm rank}}_{C/e_{\text{max}}}W$ contradicting the assumption. 

Since $C$ is regular,  there is a unique $s$ such that $C(s,l_{\text{max}})=e_{\text{max}}$. Thus both $V$ and $W$ contain the bar of the form $K(s,e_{\text{max}})$ in their bar decompositions. We can then split off this bar and proceed  by induction. 
\end{proof}

\begin{corollary}\label{asfasdfdghfgsn}
Tame  $[0,\infty)$-parametrized  $K$ vector spaces $V$ and $W$ are isomorphic if and only if  $\widehat{\text{\rm rank}}_{C}V=\widehat{\text{\rm rank}}_{C}W$  for all contours $C$.
\end{corollary}

Let us summarize our methods of producing embeddings of  isomorphism classes of tame $[0,\infty)$-parametrized  $K$ vector spaces into the space ${\mathcal M}_2$ of measurable functions of the form $[0,\infty)^2\to [0,\infty)$.  A density $f\colon[0,\infty)\to(0,\infty)$, which is a measurable function with strictly positive values, leads to two regular contours: the distance type $D_f$ (see~\ref{sdfbhndfjhgh}) and the shift type $S_f$
(see~\ref{shift type}). According to Theorem~\ref{adfgsdfg} each of these contours then leads to a sequence of pseudometrics on $\text{Tame}([0,\infty),\text{Vec}_K)$ which is ample for the rank. In this way any density leads to embeddings illustrated in the following diagrams:
\[\xymatrix@R=38pt{
*\txt{Densities \\ $f\colon[0,\infty)\to(0,\infty)$}*
\ar@/_20pt/[d]_{\text{distance type}}
\ar@/^20pt/[d]^{\text{shift type}}\\
\text{Regular contours}\ar[d]_-{\{d_{-/\alpha}\}_{\alpha\in [0,\infty]}}\\
*\txt{Sequences of pseudometrics on \\$\text{Tame}([0,\infty),\text{Vec}_K)$ ample for  \\
	$ \text{rank}\colon \text{Tame}([0,\infty),\text{Vec}_K)\to {\mathbf N}$ 
}*
}\ \ \ 
\begin{xy}
(0,0)*+{\text{Tame}([0,\infty),\text{Vec}_K)}="A",
(0,-40)*+{ {\mathcal M}_2}="B",
"A";"B" **\crv{(20,-20) }?>*\dir{>} \POS?(.4)*+!L^{\overline{\text{rank}}_{S_f}},
"A";"B" **\crv{(-20,-20) }?>*\dir{>}\POS?(.4)*+!R^{\overline{\text{rank}}_{D_f}},
\end{xy}\]
If  the density is  $1$, then the distance and the shift type contours coincide and so do the induced embeddings. 
For other densities the contours and the  embeddings are different. For example consider the constant densities $1$ and $5$, and the density $f$ displayed in Figure~\ref{fig_point_processes_contour}. In Figure~\ref{embedings} we illustrate the following functions:

\[\begin{array}{l|l}
D1:=\widehat{\text{\rm rank}}_{D_1}V =\widehat{\text{\rm rank}}_{S_1}V  &Df/0.13:=\overline{\text{rank}}_{D_f}V(0.13,-)=\widehat{\text{\rm rank}}_{D_f/0.13}V\\ 
D5:=\widehat{\text{\rm rank}}_{D_5}V &Df/0.3:=\overline{\text{rank}}_{D_f}V(0.3,-)=\widehat{\text{\rm rank}}_{D_f/0.3}V\\
S5:=\widehat{\text{\rm rank}}_{S_5}V&Sf/0.13:=\overline{\text{rank}}_{S_f}V(0.13,-)=\widehat{\text{\rm rank}}_{S_f/0.13}V\\
& Sf/0.3:=\overline{\text{rank}}_{S_f}V(0.3,-)=\widehat{\text{\rm rank}}_{S_f/0.3}V
\end{array}\]
where $V$ is a  tame $[0,\infty)$-parametrized vector space given by the first homology with ${\mathbf F}_2$ coefficients of the Vietoris-Rips construction on an IFS  point process on a unit square as described in Section~\ref{point_processes}.
\begin{figure}[!h]
\begin{minipage}{0.99\textwidth}
	\centering
	\includegraphics[width=0.48\textwidth]{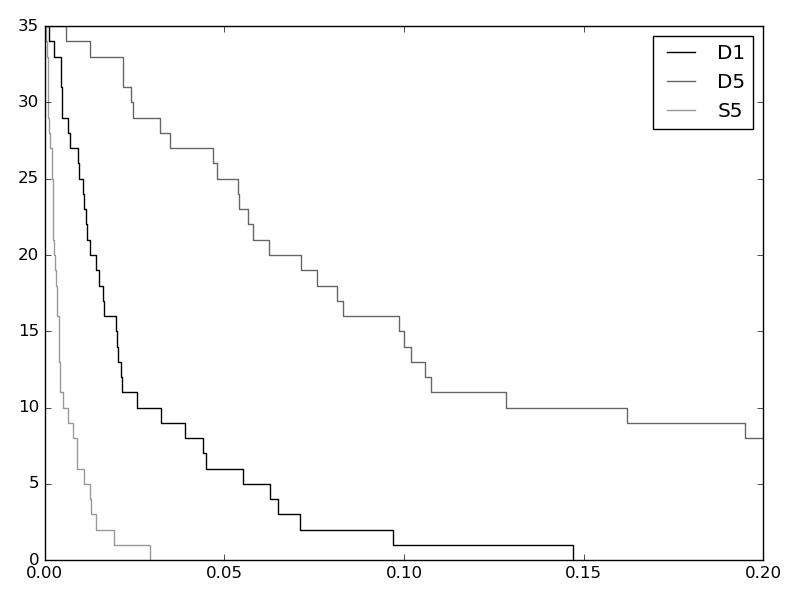}
	\includegraphics[width=0.48\textwidth]{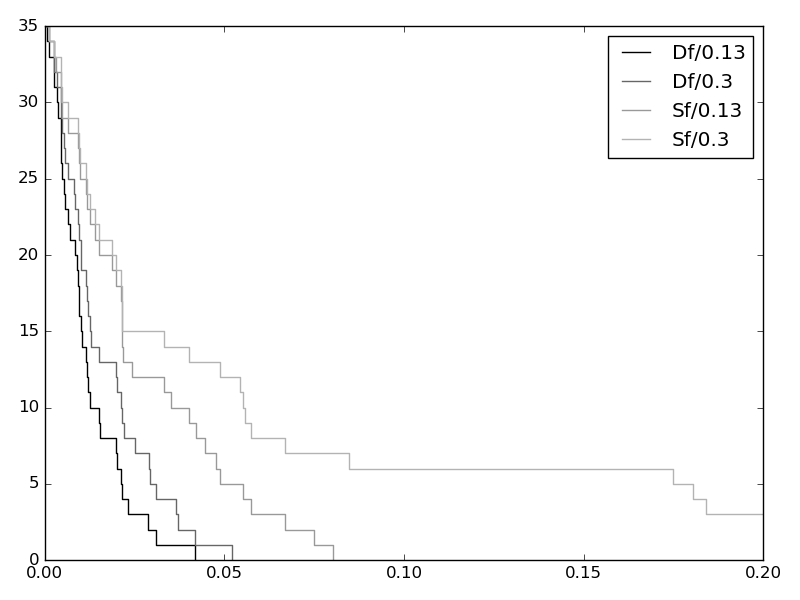}	
\end{minipage}
\caption{For  densities different than  $1$, the embeddings described after Corollary~\ref{asfasdfdghfgsn} are different.}
\label{embedings}
\end{figure}

\section{Using contours}\label{contour_usage}
In this section we illustrate how  choosing a  density and a contour can lead to improvements  in classification results based on  stable ranks. We emphasize that the focus is not on finding an optimal classifier for a specific case. The aim is to give concrete evidence to support our claim that the choice of metrics is fundamental in homological  persistence and to explain how contours are used in concrete analysis. We also hope to convey that the presented theory leads to a practical TDA pipeline amenable particularly to machine learning. Further study of the efficacy of this pipeline and the choice of contours for particular tasks is the aim of ongoing research. Two case studies are considered, point processes on a unit square and real data from human activities. To generate bars needed for our calculations we used the Ripser software \cite{Ripser} with Vietoris-Rips filtration and homology with ${\mathbf F}_2$ coefficients.

\subsection{Visualizing bars and contours}\label{contour_visualization}

The bars \(K(s,e)\) of a $[0,\infty)$-parametrized $K$ vector space can be parametrized by the start \(s\) and the life span with respect to the standard contour \(C\), \(\text{\rm life}_C K(s,e)= e-s\). Bars can then be visualized in an \((s, e-s)\)-plot as vertical stems. We call this presentation {\bf stem plot}. For a reader accustomed to barcodes, the stem plot contains exactly the same information but plotted vertically. The horizontal axis is the filtration value and the vertical axis is the bar length. Taking into account multiplicity of more than one bar having the same start value we extend the domain of the stem plot to $[0,\infty) \times \N$, where $\N$ is used to index bars with the same birth. However with real data this is needed basically only for the $0$th homology since in the standard Vietoris-Rips filtration all the points and hence all the $0$th homology classes are present at filtration value 0. 

For a fixed $t$, the relation $C(s_i,t)<e_i$ in Theorem~\ref{agfhjnd}.(3) describes an area above the parametric curve $\gamma_{t}(s_i)=(s_i,C(s_i,t))$ in the $(s, e)$-plane. Setting $C(s_i,t) = e_i$ and applying the transformation $(s,e) \mapsto (s,e-s)$, we get a curve $\widehat{\gamma}_{t}(s_i)=(s_i,C(s_i,t)-s_i)$. Such curves are typically called contour lines, hence the name contours.

Significance of stem plots comes when overlaid with contour lines. The right plot of Figure \ref{fig_contours_use} illustrates a persistence stem plot along with contour lines of distance and shift contours for few values of \(t\) (dashed curves). Stem plot comes from one realization of point processes of Section \ref{point_processes}. Density function used to calculate contours is also shown in the left plot. With contour lines the vertical axis of a stem plot corresponds to \(t\) in $C(s,t)$ and the horizontal axis is the filtration value \(s\) as explained above.

Stem plot and contour lines make it easy to understand visually Proposition~\ref{asfdsdfgfd}: the value of the stable rank at $t$ is the number of those bars that exceed the contour line at \(t\). Thus on those regions where contour lines obtain low values, homological features are magnified and vice versa for larger values of contour lines. Stem plot can be an effective tool to gain understanding of stable ranks with respect to different contours and to explore appropriate ones for a given task. 

\begin{figure}
	\begin{minipage}{0.99\textwidth}
		\includegraphics[width=0.49\textwidth]{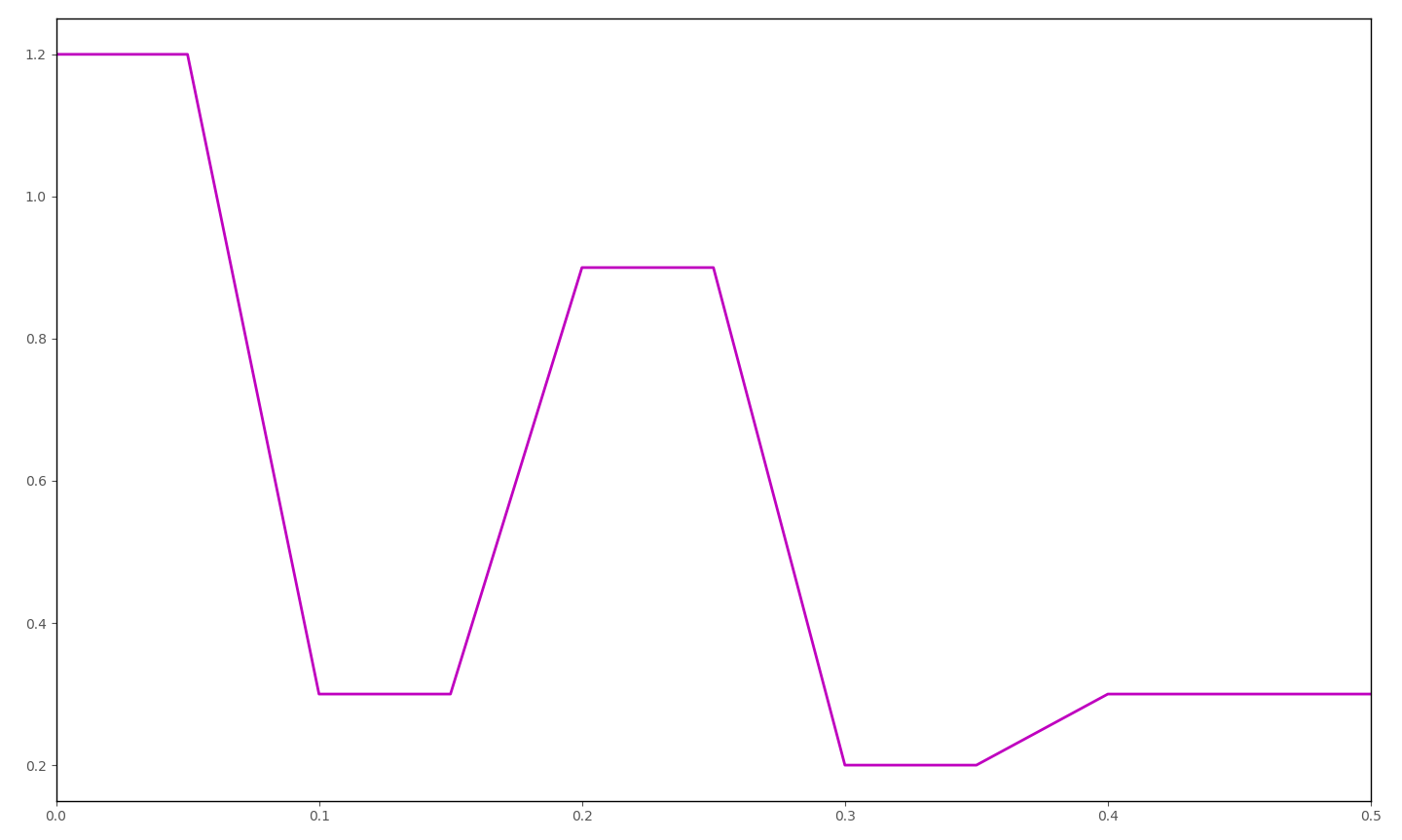}
		\includegraphics[width=0.49\textwidth]{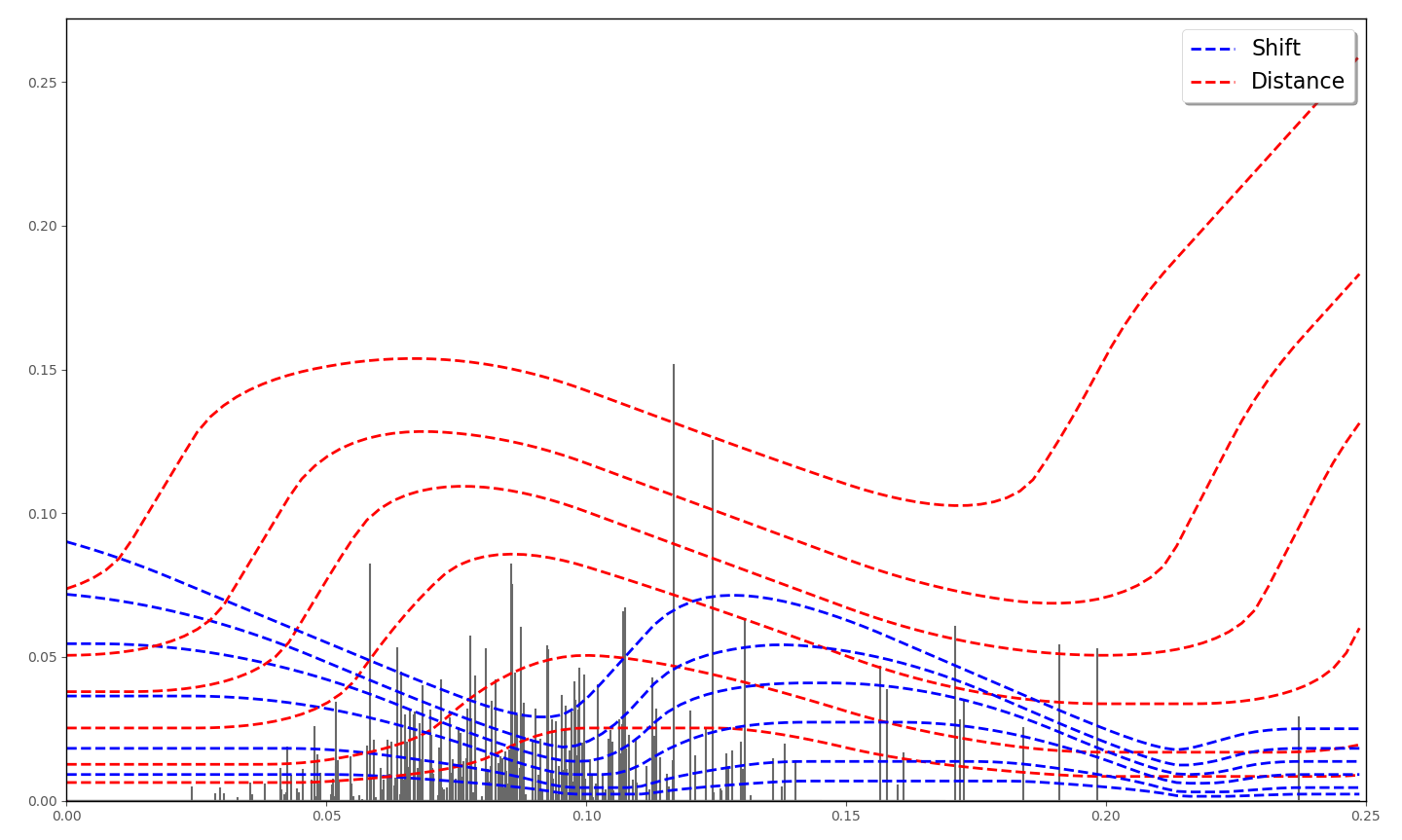}	
	\end{minipage}
	\caption{Distance and shift type contours visualized on a stem plot (right) with the same density function (left).}
	\label{fig_contours_use}
\end{figure}

\subsection{Point processes}\label{point_processes}
Point processes have gathered interest in TDA community, see for example \cite{APF, LimitTheorems, HypothesisTesting}. We simulated six different classes of point processes on a unit square, see their descriptions below. For each class we produced 500 simulations on average containing 200 points. Let $X \sim PD(k)$ denote that random variable $X$ follows probability disribution $PD$ with parameter $k$. In particular, $\text{Poisson}(\lambda)$ denotes the Poisson distribution with event rate $\lambda$.
\smallskip

\noindent
\textbf{Poisson}: We first sampled number $N$ of events, where $N \sim \text{Poisson}(\lambda)$. We then sampled $N$ points from a uniform distribution defined on the unit square $[0,1] \times [0,1]$. Here $\lambda=200$.
\smallskip

\noindent
\textbf{Normal}: Again number $N$ of events was sampled from $\text{Poisson}(\lambda)$, $\lambda = 200$. We then created $N$ coordinate pairs $(x,y)$, where both $x$ and $y$ are sampled from normal distribution $N(\mu,\sigma^2)$ with mean $\mu$ and standard deviation $\sigma$. Here $\mu=0.5$ and $\sigma = 0.2$.
\smallskip

\noindent
\textbf{Matern}: Poisson process as above was simulated with event rate $\kappa$. Obtained points represent parent points, or cluster centers, on the unit square. For each parent, number $N$ of child points  was sampled from $\text{Poisson}(\mu)$. A disk of radius $r$ centered on each parent point was defined. Then, for each parent, the corresponding number  $N$ of child points were placed on the disk. Child points were distributed by a uniform distributions on the disks. Note that parent points are not part of the actual data set. We set $\kappa$=40, $\mu$=5, and $r=0.1$.
\smallskip

\noindent
\textbf{Thomas}: Thomas process is similar to Matern process except that instead of uniform distributions, child points are sampled from bivariate normal distributions defined on the disks. The distributions were centered on the parents and had diagonal covariance $\bigl[\begin{smallmatrix} 0.1^2 & 0 \\ 0 & 0.1^2 \end{smallmatrix} \bigr]$. \smallskip

\noindent
\textbf{Baddeley-Silverman}: For this process the unit square was divided into equal size tiles with side lengths $1/14$. Then for each tile, $N$ points were sampled, $N \sim \text{Baddeley-Silverman}$. Baddeley-Silverman distribution is a discrete distribution defined on values $(0,1,10)$ with  probabilities  $(\frac{1}{10},\frac{8}{9},\frac{1}{90})$. For each tile, associated number of  $N$ points were then uniformly distributed on the tile.
\smallskip

\noindent
\textbf{Iterated function system (IFS)}: We also generated point sets with an iterated function system. For this a discrete distribution is defined on values $(0,1,2,3,4)$ with corresponding probabilities 
$\left(\frac{1}{3},\frac{1}{6},\frac{1}{6},\frac{1}{6},\frac{1}{6}\right)$. We denote this distribution by IFS. Starting from an initial point $(x_0,y_0)$ on the unit square, $N \sim \text{Poisson}(200)$ new points were generated by the recursive formula $(x_n,y_n) = f_i(x_{n-1},y_{n-1}),$ where $n \in \{1,...,N\}$, $i \sim \text{IFS}$ and the functions $f_i$ are given as
\[f_0(y,x) = \left(\frac{x}{2},\frac{y}{2}\right), f_1(y,x)= \left(\frac{x}{2}+\frac{1}{2},\frac{y}{2}\right), 
f_2(y,x) = \left(\frac{x}{2},\frac{y}{2}+\frac{1}{2}\right),\] 
\[f_3(y,x) = \left(\left|\frac{x}{2}-1\right|,\frac{y}{2}\right), 
f_4(y,x)= \left(\frac{x}{2},\left|\frac{y}{2}-1\right|\right). 
\]
\smallskip

Figure \ref{fig_point_processes} shows realizations of the point processes with given parameters. From topological data analysis point of view the point sets hold no distinct large scale topology. It is therefore ideal to study the geometric correlations or features in the filtration captured by homologies in degrees 0 and 1, denoted $H_0$ and $H_1$ respectively.

\begin{figure}
	\begin{minipage}{0.99\textwidth}
		\centering
		\includegraphics[width=0.32\textwidth]{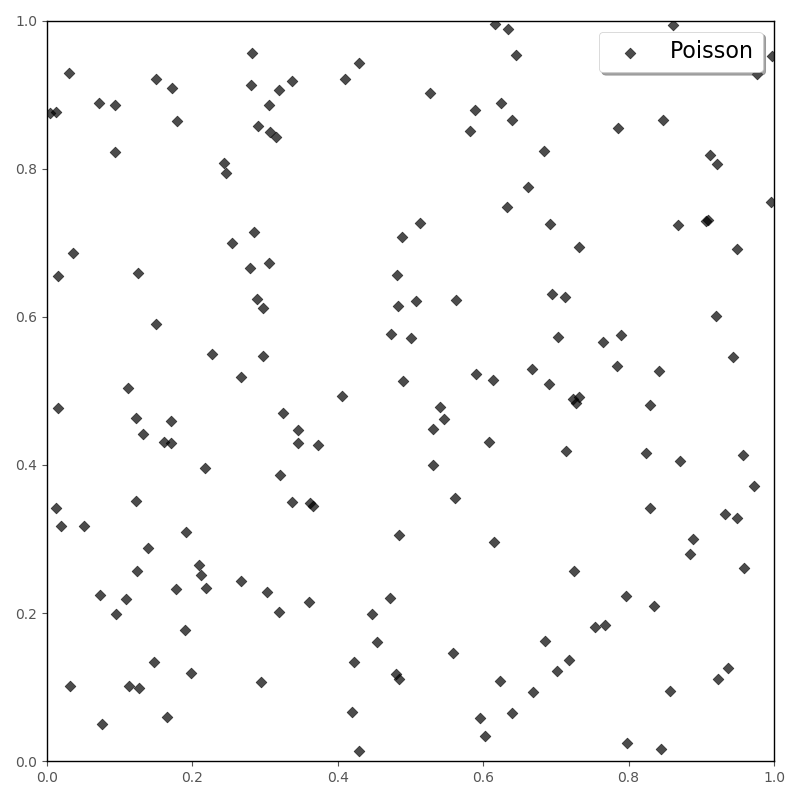}
		\includegraphics[width=0.32\textwidth]{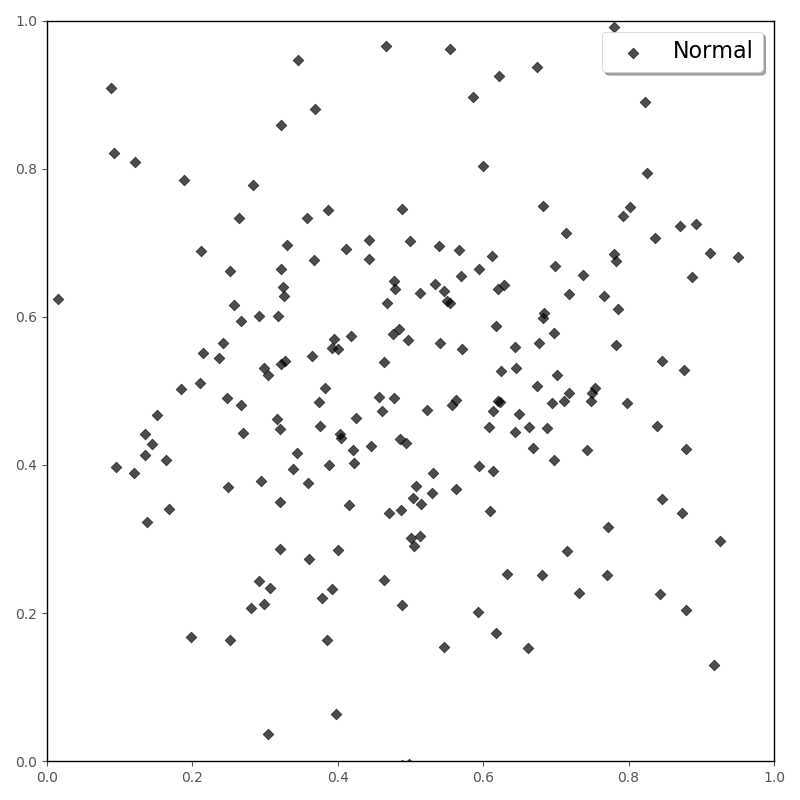}
		\includegraphics[width=0.32\textwidth]{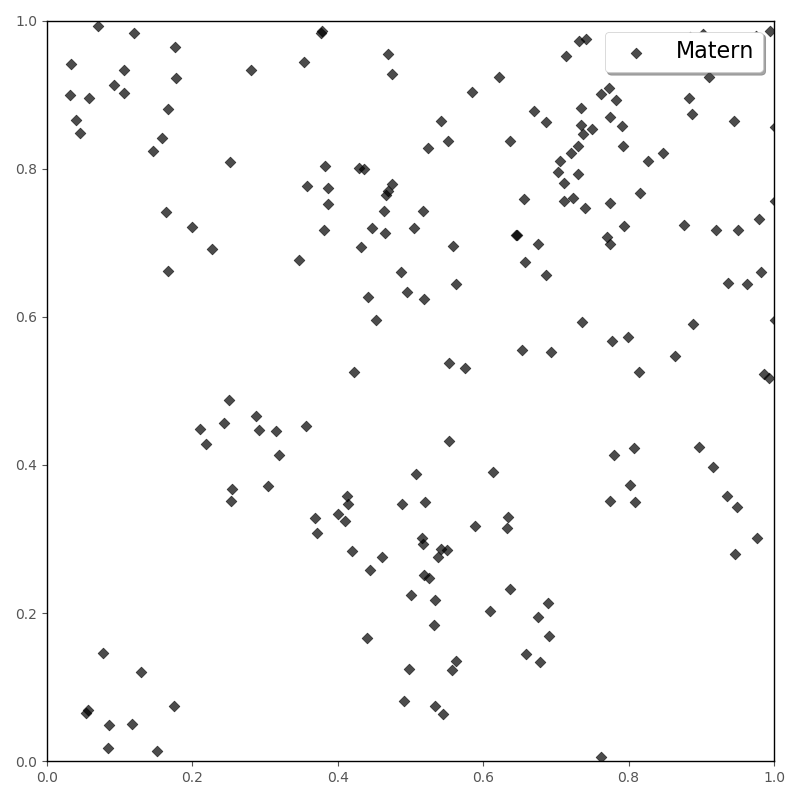}	
	\end{minipage}
	\begin{minipage}{0.99\textwidth}
		\centering
		\includegraphics[width=0.32\textwidth]{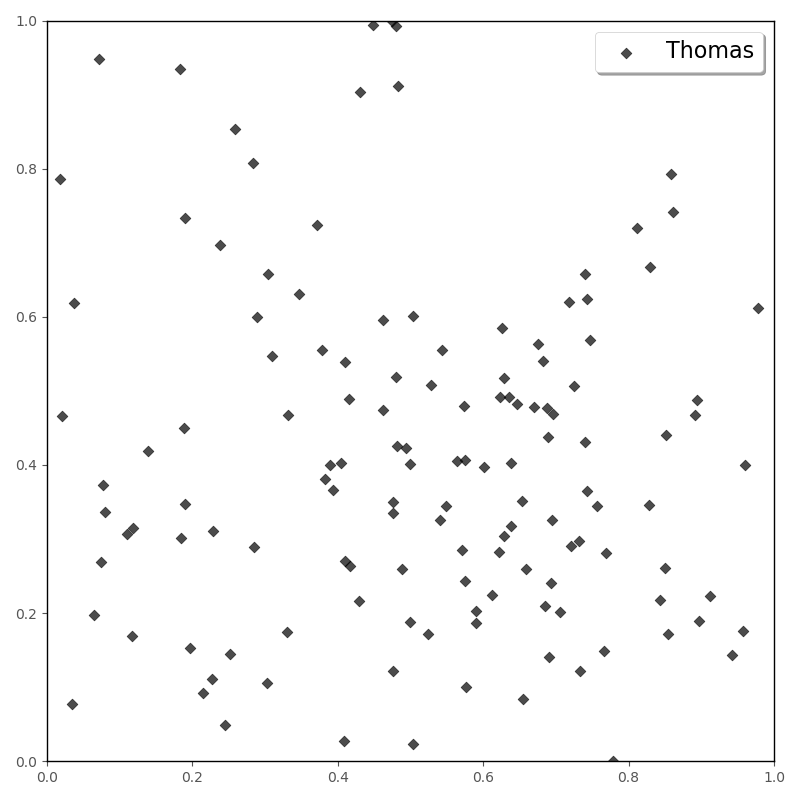}
		\includegraphics[width=0.32\textwidth]{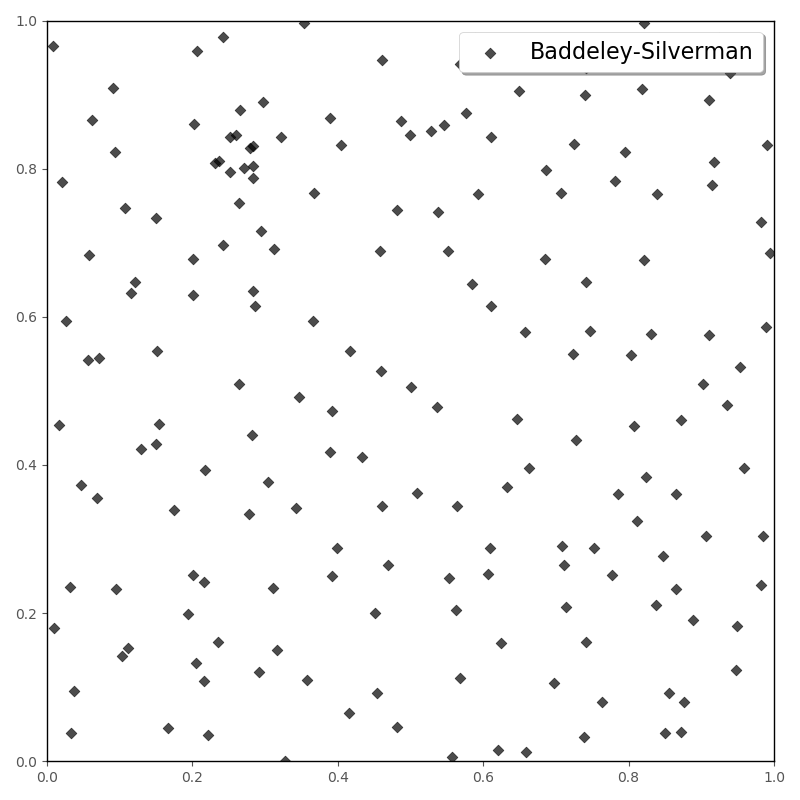}
		\includegraphics[width=0.32\textwidth]{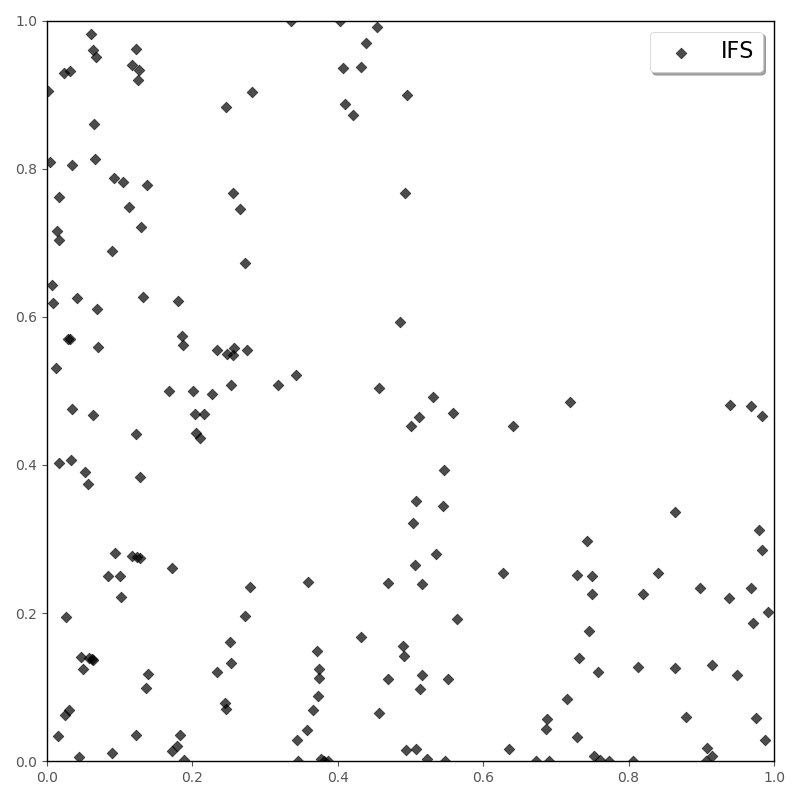}	
	\end{minipage}
	\caption{Example realizations of point processes on unit square.}
	\label{fig_point_processes}
\end{figure}

Figure \ref{fig_PointProc_stable_ranks_dist_shift} shows $H_1$ stable ranks from distance and shift contours for one realization of the point processes. Corresponding stem plot and contour lines are shown in Figure \ref{fig_contours_use}. Note the different character of stable ranks between contours. Distance contour decreases lifespans of bars relative to it making the stable ranks decrease to zero faster and also diminishing their separation (Figure \ref{fig_PointProc_stable_ranks_dist_shift} left). Comparing, for example, Poisson and Baddeley-Silverman point processes in Figure \ref{fig_point_processes}, Poisson seems to have larger \(H_1\) features appearing at larger filtration values. The point structure of Baddeley-Silverman seems to indicate that it has smaller \(H_1\) features at smaller filtration values. Shift contour increases lifespans of those smaller Baddeley-Silverman bars around filtration value \(s=0.08\) (right plot in Figure \ref{fig_contours_use}) whereas more of the later Poisson bars are shortened by the contour after \(s=0.08\). This can be seen in Baddeley-Silverman dominating Poisson stable rank (Figure \ref{fig_PointProc_stable_ranks_dist_shift} right). Note that the horizontal axes in Figure \ref{fig_PointProc_stable_ranks_dist_shift} correspond to the \(t\) variable of contour $C(s,t)$ while the horizontal axes of stem plots are the filtration values \(s\).

Explanation in the previous paragraph exemplifies how the choice of metric allows analyst to emphasize differently homological features in persistence analysis. As referenced in Section \ref{adfgafdghsfgh}, various recent applications have shown that bars of different sizes and also their locations in the filtration might be deemed important for a given analysis task. The framework of hierarchical stabilization facilitates this kind of exploration. For instance, consider set of samplings of some dynamical phenomenon. Analysis with contours might help understanding whether observed smaller \(H_1\) features are just sampling noise or indicate actual puncturing of the underlying topology of the dynamics. Quantifying importance of holes is also interesting in relational databases where they indicate missing data values or non-allowed attribute combinations \cite{BigHoles}. 

\begin{figure}
	\begin{minipage}{0.99\textwidth}
		\includegraphics[width=0.49\textwidth]{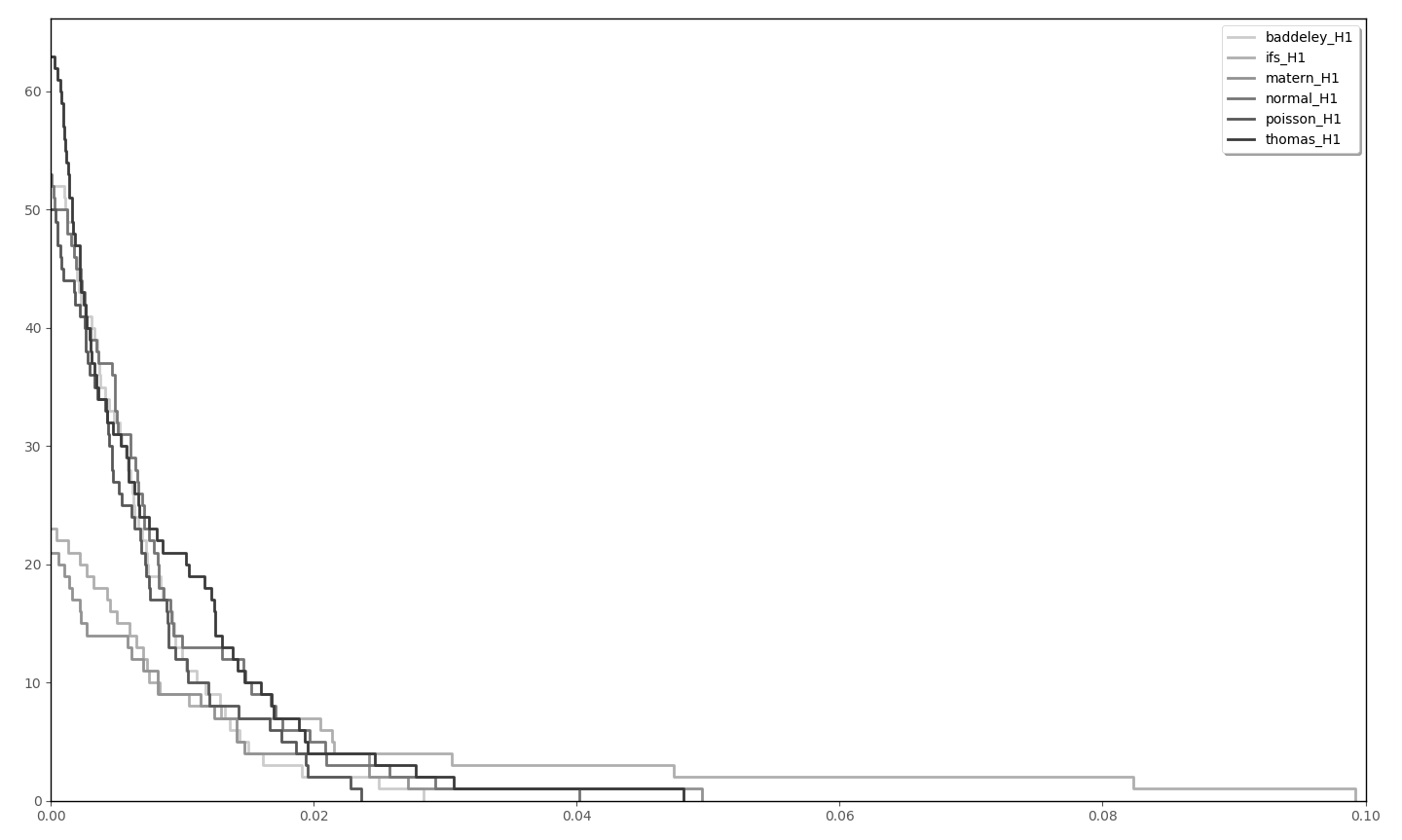}
		\includegraphics[width=0.49\textwidth]{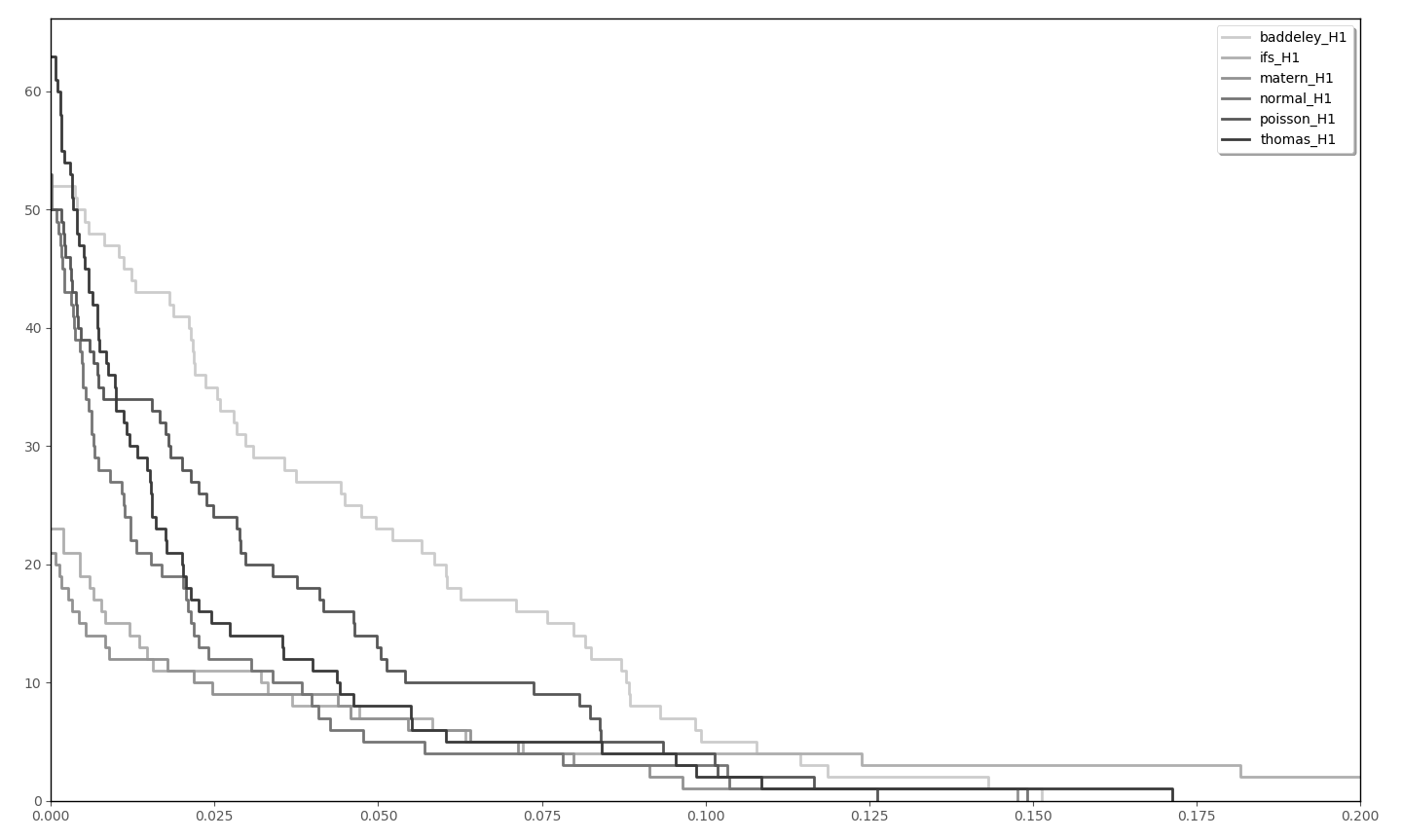}	
	\end{minipage}
	\caption{Stable ranks in \(H_1\) of point processes for distance (left) and shift (right) contours of Figure \ref{fig_contours_use}. The bars are also shown in the stem plot of Figure \ref{fig_contours_use}.}
	\label{fig_PointProc_stable_ranks_dist_shift}
\end{figure}

\begin{figure}
	\begin{minipage}{0.99\textwidth}
		\includegraphics[width=0.32\textwidth]{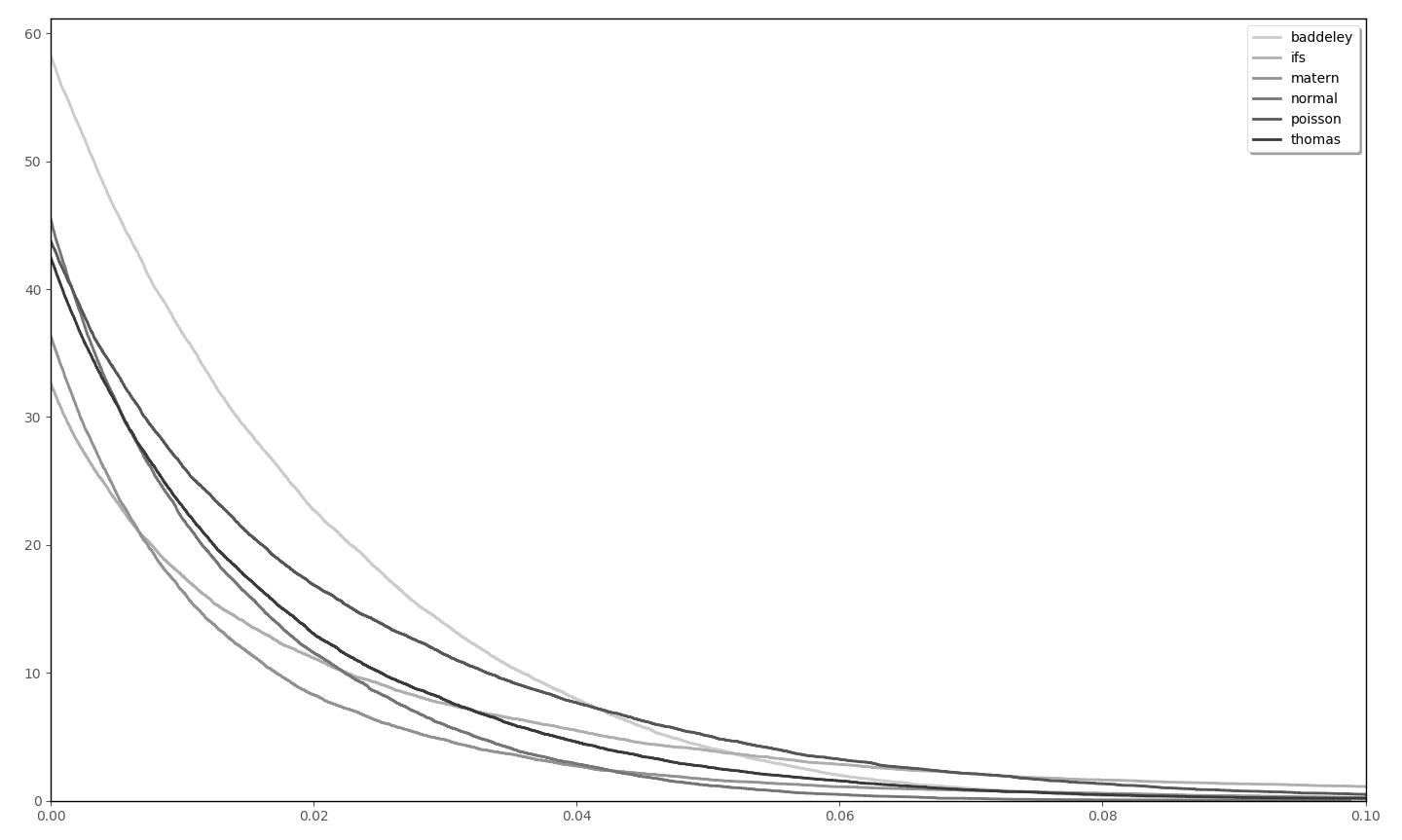}
		\includegraphics[width=0.32\textwidth]{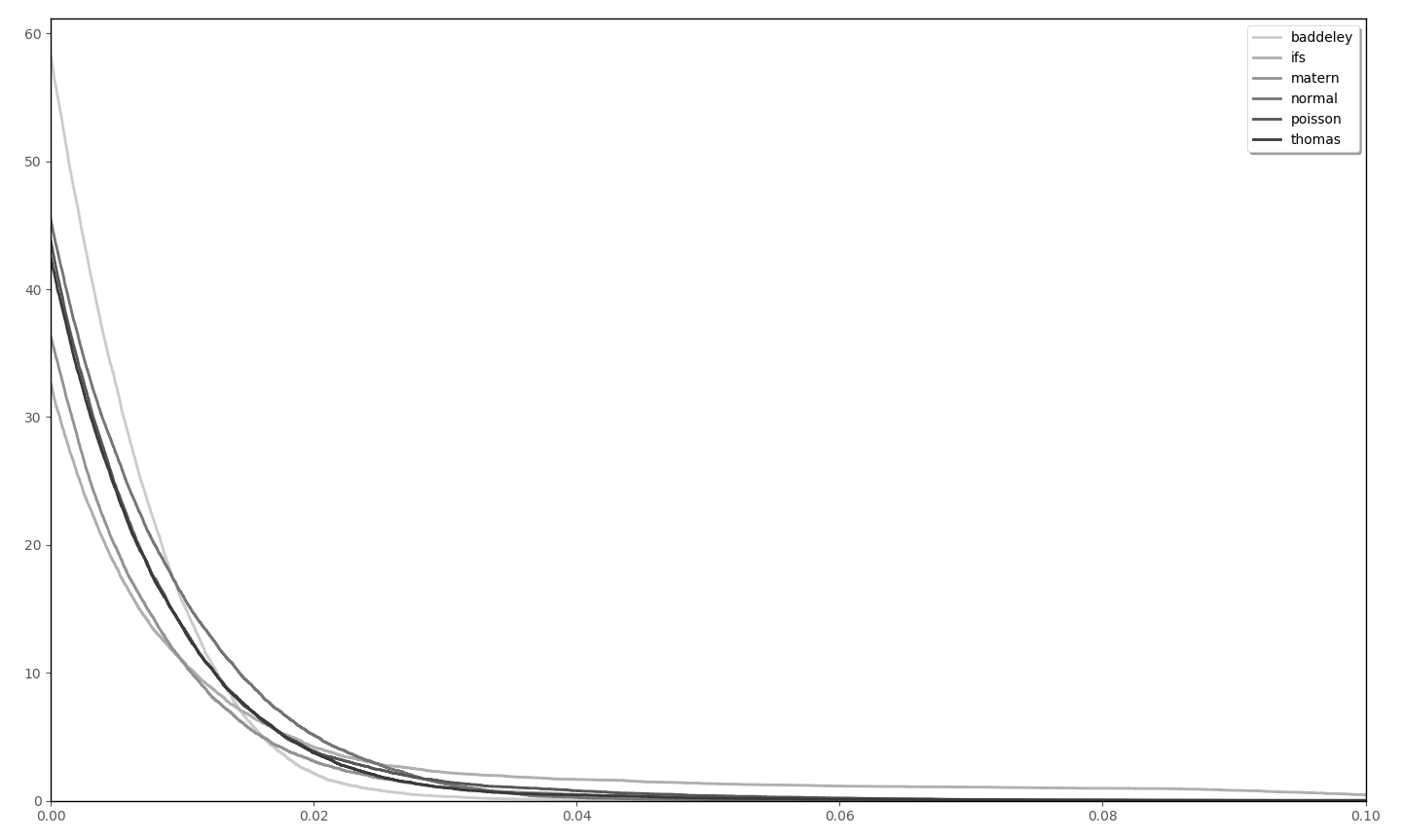}
		\includegraphics[width=0.32\textwidth]{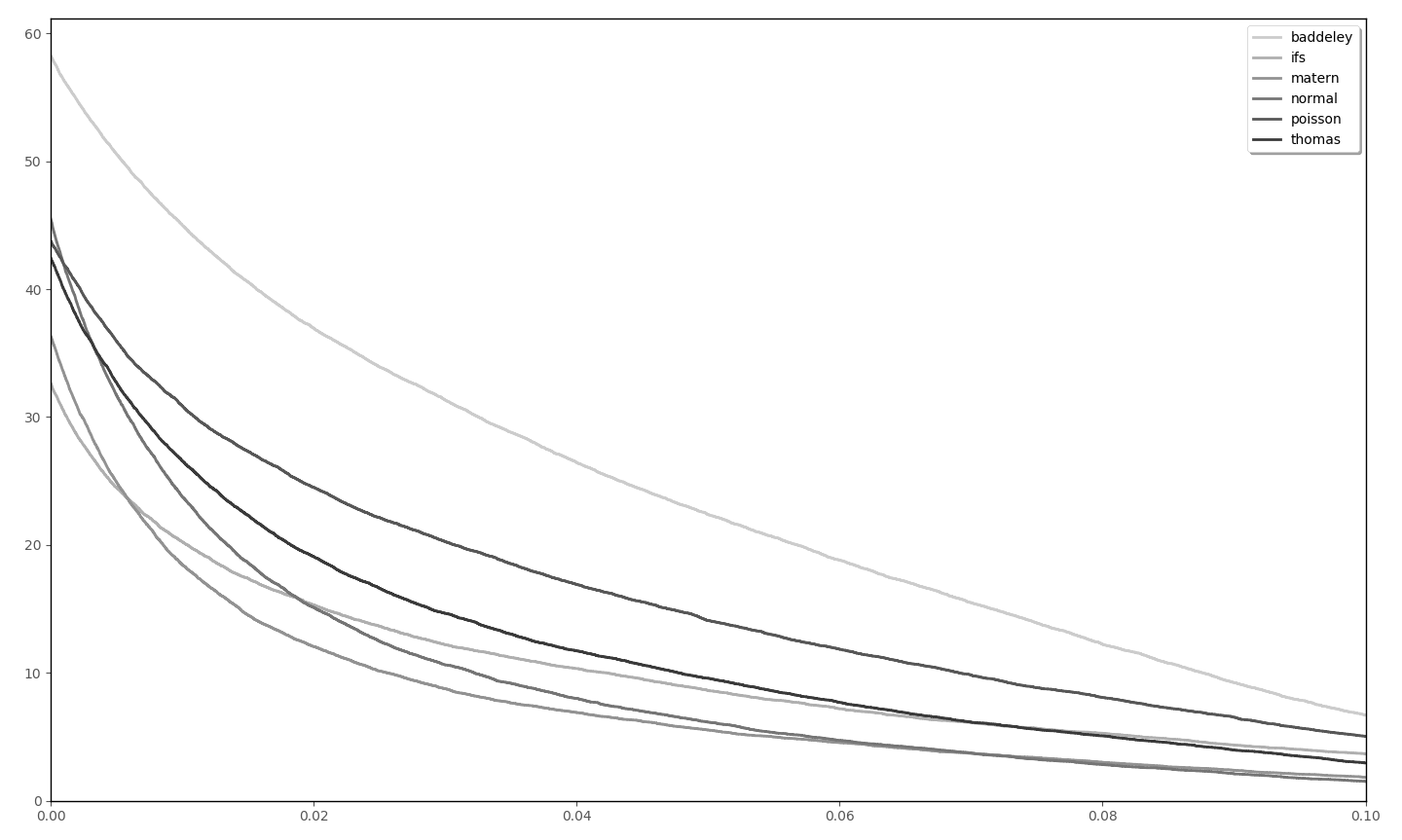}
	\end{minipage}
	\caption{Mean $H_1$ stable ranks of 200 simulations of point processes with respect to the standard contour (left), distance contour of Figure \ref{fig_contours_use} (middle) and shift contour  of Figure \ref{fig_contours_use} (right). All the plots are on the same scale.}
	\label{fig_point_processes_avg200_stan}
\end{figure}

Figure \ref{fig_point_processes_avg200_stan} is a plot of the averages (point-wise means) of $H_1$ stable ranks  with respect to the standard contour and distance and shift contours of Figure \ref{fig_contours_use} for 200 simulations of the point processes. Shift contour increases the separation between the stable ranks as compared to the standard contour, whereas distance contour decreases the separation. It is worth noting that Matern and Thomas processes are well distinguished by the shift contour in the right plot even though in their definition they only differ in the distribution used for point clusters. 

To test how well the stable ranks with respect to different contours perform in classifying different point processes we conducted mean classification procedure:

\begin{itemize}
	\item For each class choose 200 simulations as a training set. Remaining 300 simulations form test set  for the class.
	\item Compute the point-wise means of the training set stable ranks with respect to the chosen contour. These mean invariants are used as classifiers, denoted by \(\widehat{C}_{H_\bullet}\), where \(H_\bullet\) refers to the corresponding homology.
	\item Denote stable ranks in the test set by \(T_{H_\bullet}\). Compute distances \(L_1(\widehat{C}_{H_\bullet},T_{H_\bullet})\) between each test element and all classifiers.
	\item Record found minimum distance by adding 1 to the corresponding pair of the classifier and the test  class. Classification is successful if the classifier and the test  belong to the same class (in the optimal case the value of the pair (Poisson \(\widehat{C}\), Poisson \(T\)) would be 300, for example). 
	\item For cross-validation use 20-fold random subsampling. Randomly sample 200 stable ranks for classifiers, remaining 300 stable ranks in each class constitute the test sets. Repeat the classification procedure above 20 times and take the classification accuracy to be the average over the folds. 
\end{itemize} 
\smallskip

Cross-validated classification accuracies with standard contour are reported in the confusion matrices of Figure \ref{fig_point_proc_classification_acc_integral_stan}. The confusion matrices show relative accuracies after dividing by 300 after each fold and averaging after the full cross-validation run. The mean classification accuracy by taking the average over classes (average of the diagonal) is 85\% for $H_0$ and 73\% for $H_1$. The classification procedure performs comparably or better as the hypothesis testing against the homogeneous Poisson process in \cite{APF}. Note that no other assumptions or parameter selections were involved in our methodology other than the split between training and test samples (200 and 300, respectively.)

\begin{figure}
	\begin{minipage}{0.99\textwidth}
		\includegraphics[width=0.49\textwidth,clip,trim=7cm 2cm 7.3cm 1.2cm]{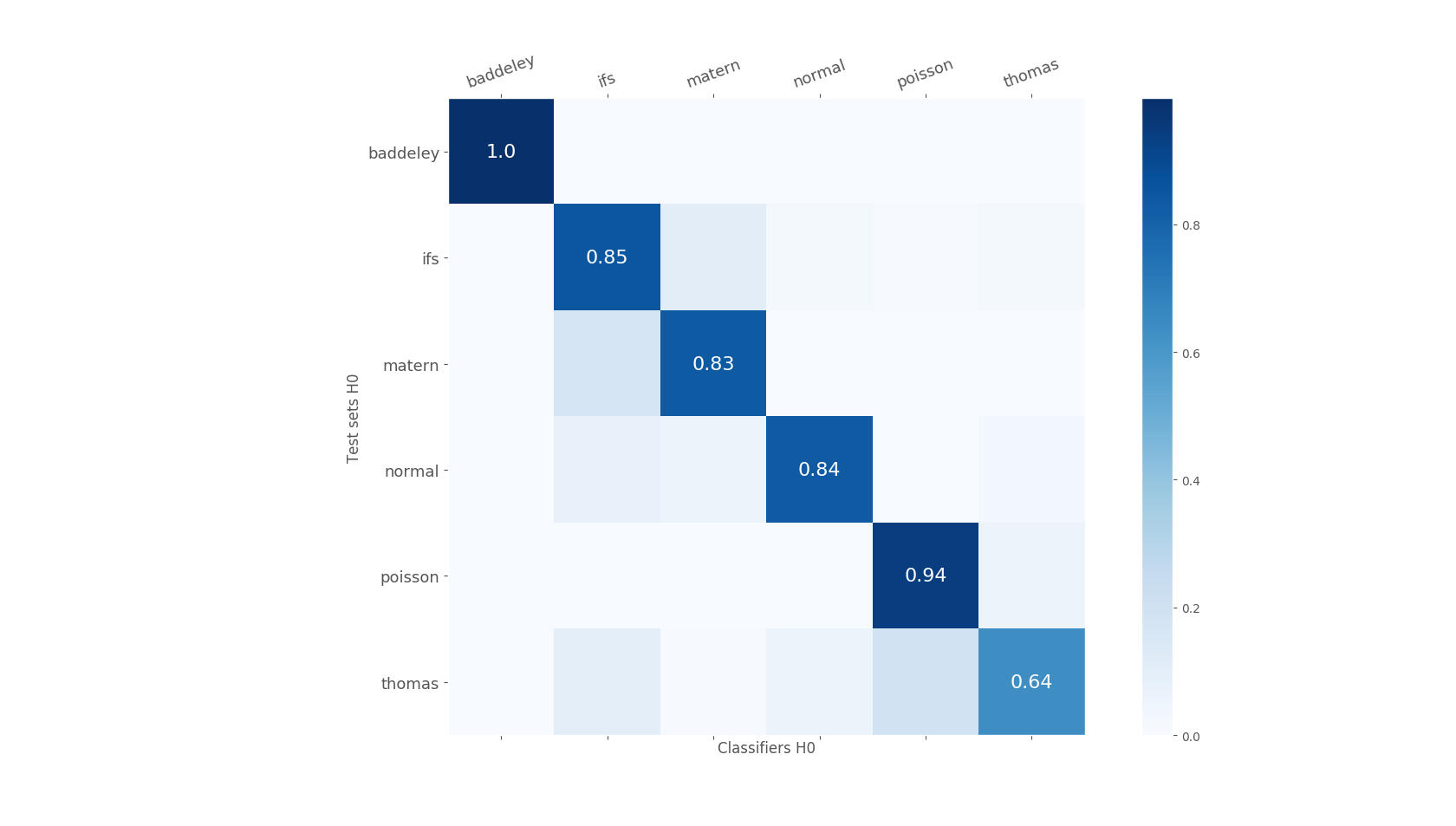}
		\includegraphics[width=0.49\textwidth,clip,trim=7cm 2cm 7.3cm 1.2cm]{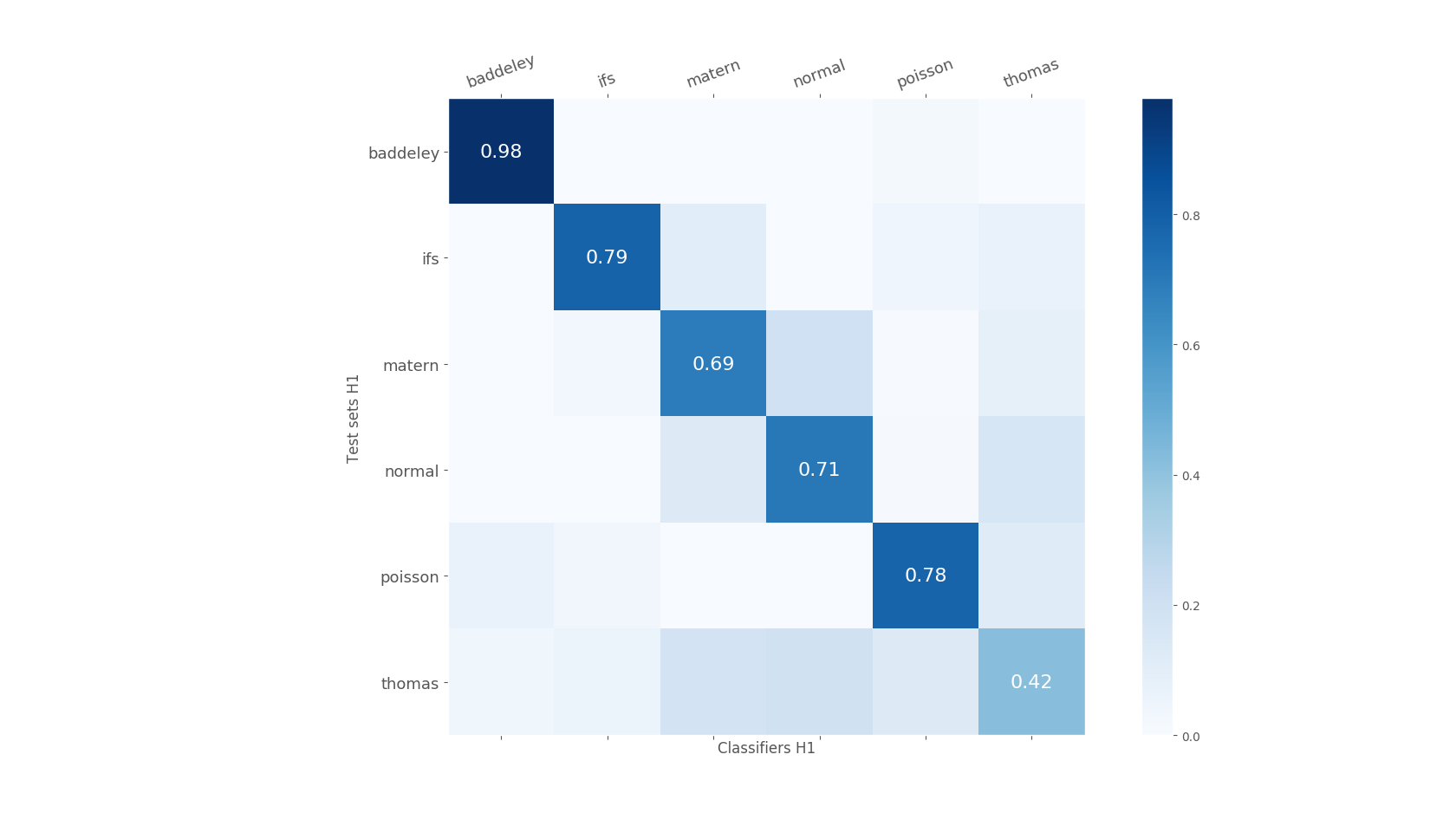}
	\end{minipage}
	\caption{Confusion matrices for the point process classification in $H_0$ (left) and $H_1$ (right) with standard contour.}
	\label{fig_point_proc_classification_acc_integral_stan}
\end{figure}

Figure \ref{fig_crossval_classification_accuracies_H1_integral_contoured_200_78} shows cross-validated classification accuracies for $H_1$ stable ranks with shift contour described in Figure \ref{fig_point_processes_contour}. We thus increase the lifespans of features appearing in the middle of the filtration. The overall classification accuracy increased to 78\%. Particularly classification accuracy of the Thomas process was drastically improved as shown in the confusion matrix of Figure \ref{fig_crossval_classification_accuracies_H1_integral_contoured_200_78}. Also noteworthy is the improvement in the accuracy of normal and Poisson processes. Using the shift contour thus captures more relevant distinguishing homological information of the point processes compared to the standard contour. 

\begin{figure}[!h]
	\begin{minipage}{0.99\textwidth}
		\centering
		\includegraphics[width=0.49\textwidth]{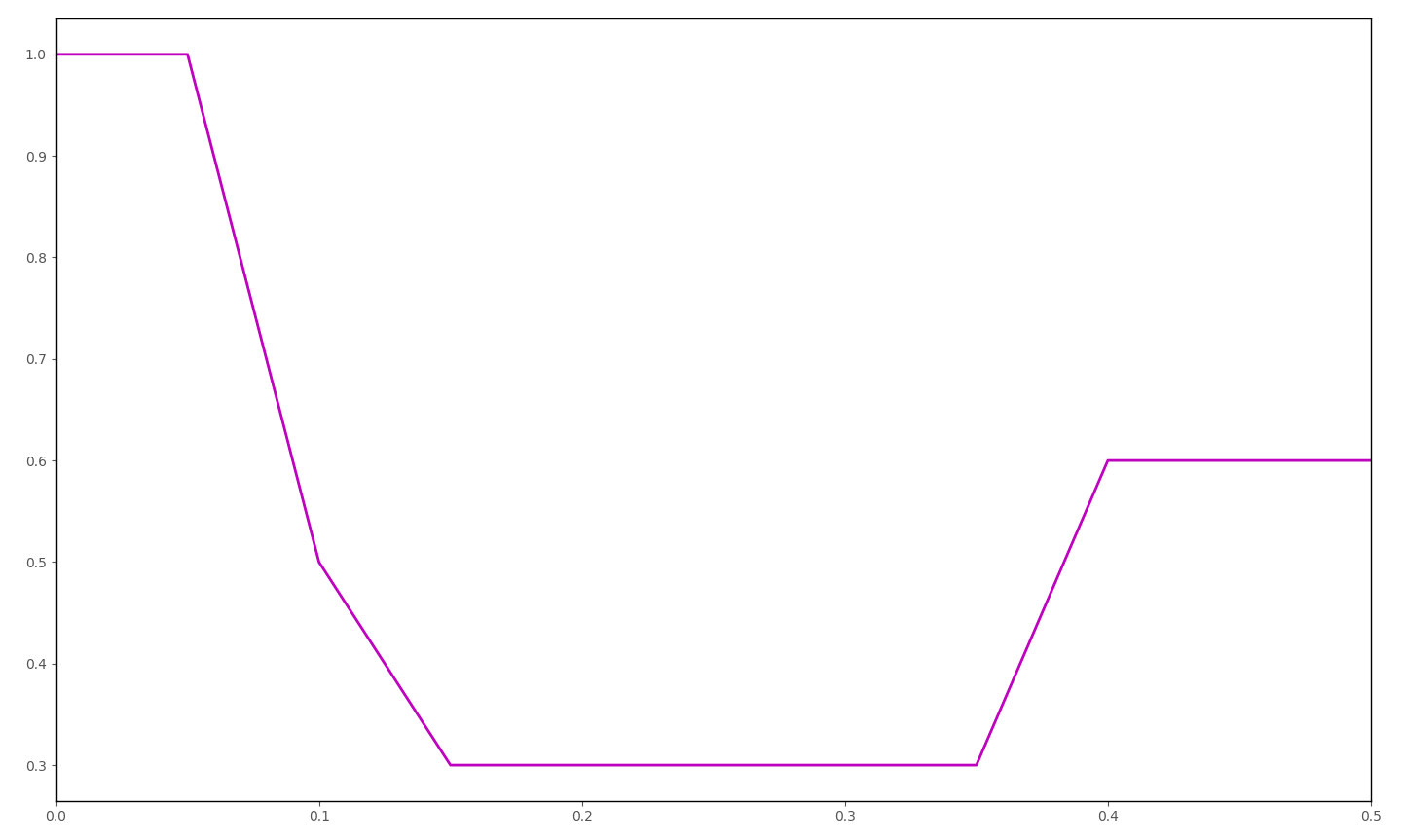}
		\includegraphics[width=0.49\textwidth]{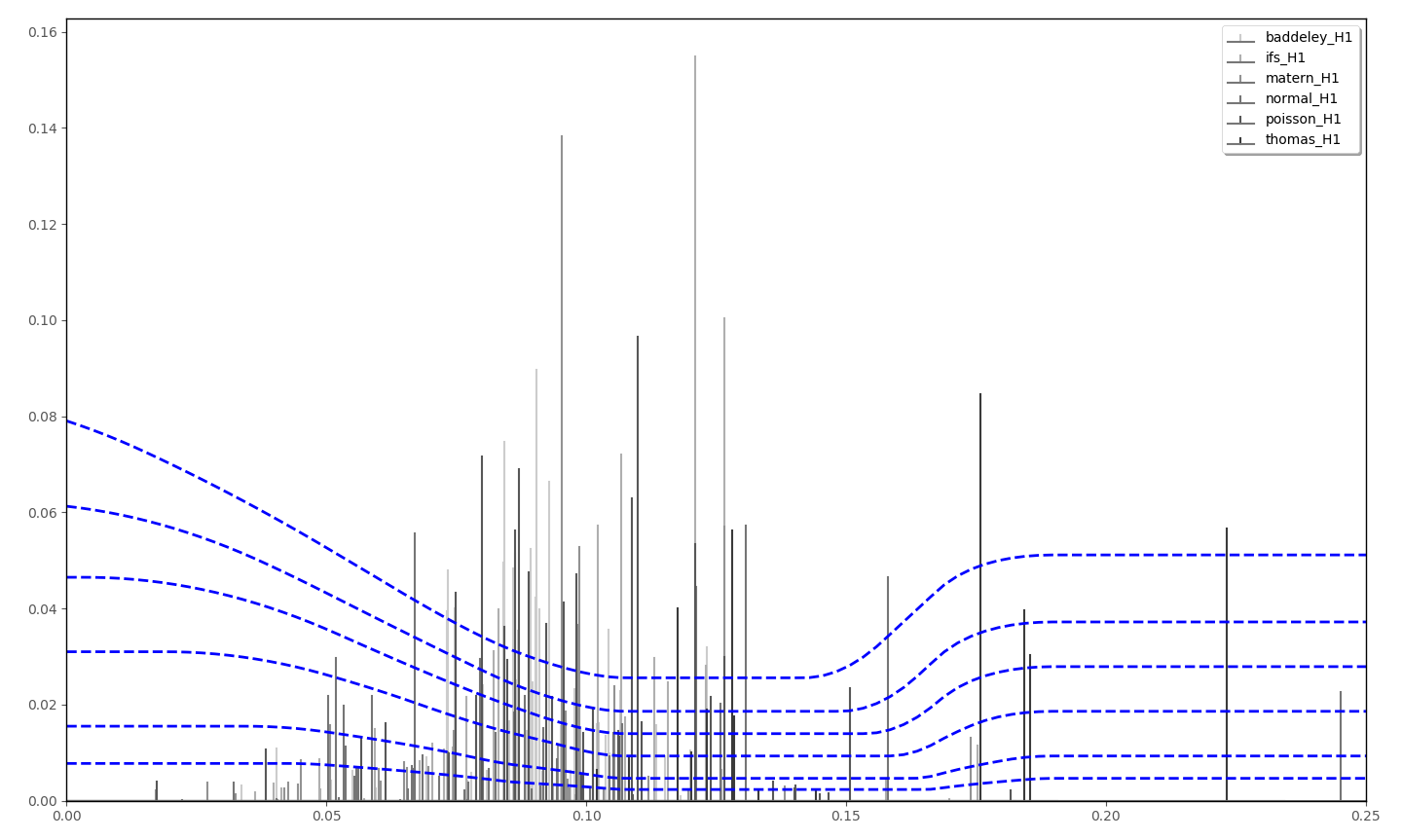}
	\end{minipage}
	\caption{Density function used in producing shift contour for point process classification in $H_1$ (left). Corresponding contour lines and stem plots from $H_1$ persistence analysis of one realization of the studied point processes (right).}
	\label{fig_point_processes_contour}
\end{figure}

\begin{figure}
	\centering
	\includegraphics[scale=0.27,clip,trim=7cm 2cm 2cm 1.2cm]{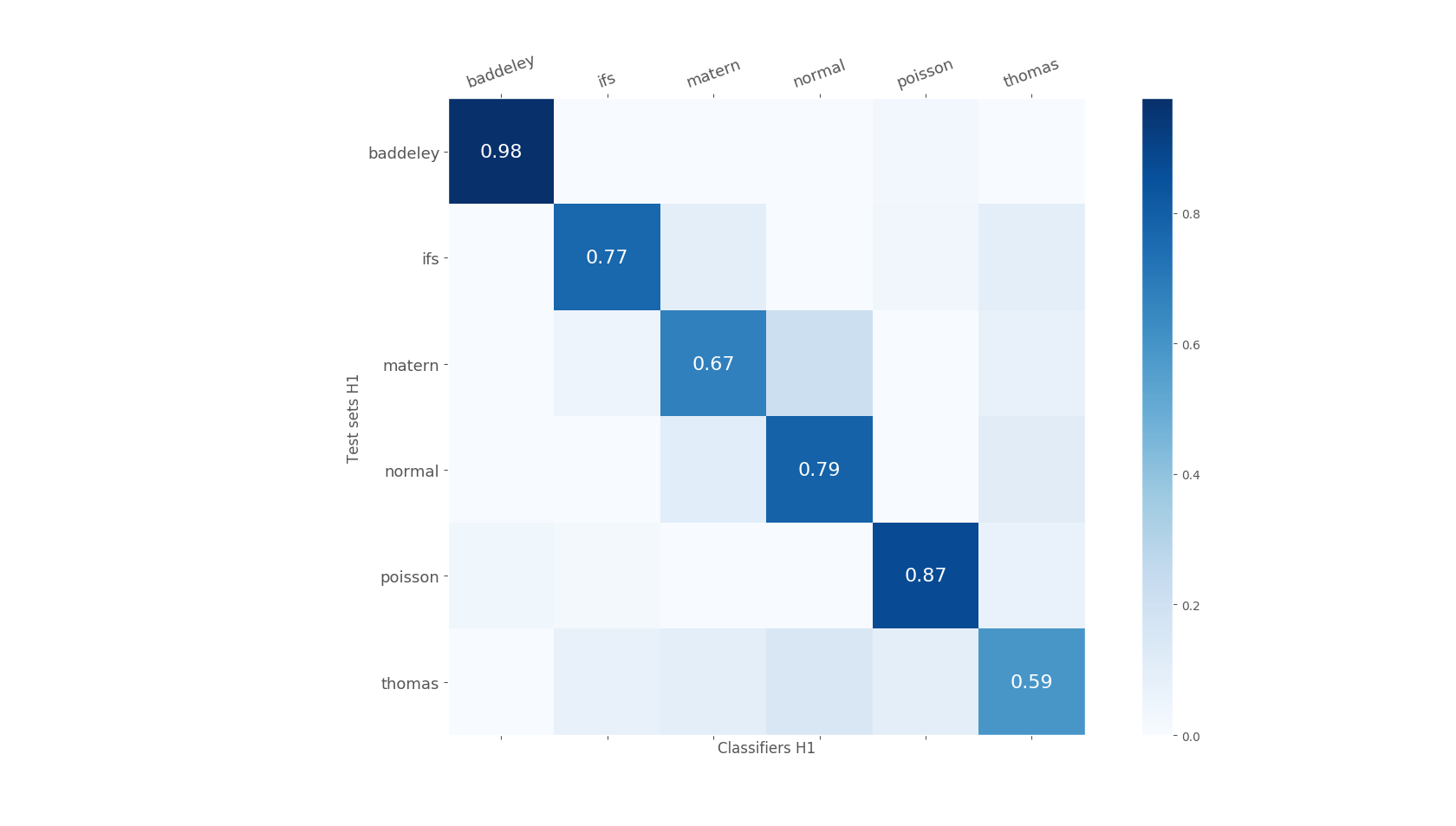}
	\caption{Confusion matrix for the classification of point processes in $H_1$ with contour coming from density function in Figure \ref{fig_point_processes_contour}.}
	\label{fig_crossval_classification_accuracies_H1_integral_contoured_200_78}
\end{figure}

\subsection{Activity monitoring}\label{FASGDFSG}
As an application to real data we studied activity monitoring of different physical activities. Used data set was PAMAP2 data obtainable from~\cite{PAMAP}. It makes sense to use all the persistence information, i.e.\  to combine homologies of different degrees into single classification scheme. In this section we demonstrate how this is enabled by stable ranks and our pipeline. 

The data consisted of seven persons from the PAMAP2 data set performing different activities such as walking, cycling, vacuuming and sitting. Test subjects were fitted with three Inertial Measurements Units (IMUs), one on wrist, ankle and chest, and a heart rate monitor. Measurements were registered every 0.1 seconds. Each IMU measured 3D acceleration, 3D gyroscopic and 3D magnetometer data. One data set thus consisted of 28-dimensional data points indexed by 0.1 second timesteps. 

We looked at two activities in this case study: ascending and descending stairs. At the outset one would expect these activities to be very similar and therefore difficult to distinguish. For persistence analysis we randomly sampled without replacement 100 points from each data set, repeated 100 times. For each of the 7 subjects we thus obtained 100 resamplings from their activity data. We computed $H_0$ and $H_1$ persistence for each sampling. The classification procedure was the same as outlined in Section \ref{point_processes} except we combined both homologies in the classifier as follows. We took the mean of 40 out of 100 stable ranks both in $H_0$ and $H_1$. We thus obtained 14 classifier pairs $(\widehat{C}_{H_0},\widehat{C}_{H_1})$ corresponding to all (subject, activity) classes. Remaining 60 stable ranks formed test data pairs $(T_{H_0}, T_{H_1})$ in each class. For a  pair $(T_{H_0}, T_{H_1})$ we then considered 
\[\text{min}\{L_1(\widehat{C}_{H_0},T_{H_0}) + L_1(\widehat{C}_{H_1},T_{H_1})\} \text{ for each pair }(\widehat{C}_{H_0},\widehat{C}_{H_1}).\]
Again the classification is successful if the minimum is obtained with $(\widehat{C}_{H_0},\widehat{C}_{H_1})$ and $(T_{H_0}, T_{H_1})$ belonging to the same (subject, activity) class.

Results for 20-fold random subsampling cross-validation are shown on the left in Figure \ref{fig_activities_confusions} for the standard contour, with the overall accuracy of 60\%. For the classification with a different contour we use the standard contour for $H_0$ and the shift contour of Figure \ref{fig_activities_density_contour} for \(H_1\). The results are shown on the right in Figure \ref{fig_activities_confusions}. The shift contour increases lifespans of \(H_1\) features appearing with larger filtration values. Exploring the stem plots (Figure \ref{fig_activities_density_contour}) for different data sets shows that larger filtration values have bars sparsely (some data sets having no bars) and their lengths vary significantly between different classes of data. This observation leads to use the contour emphasizing bars in the larger filtration values, by which the classification accuracy increases to 65\%. Note particularly increase in the accuracy of subject 4. Also noteworthy is that ascendings mainly get confused with ascendings and the same for descendings. These data thus exhibit clearly different character and using an appropriate contour makes this difference more pronounced. 

\begin{figure}[!h]	
	\begin{minipage}{0.99\textwidth}
		\centering
		\includegraphics[scale=0.38,clip,trim=5cm 1.2cm 4cm 0cm]{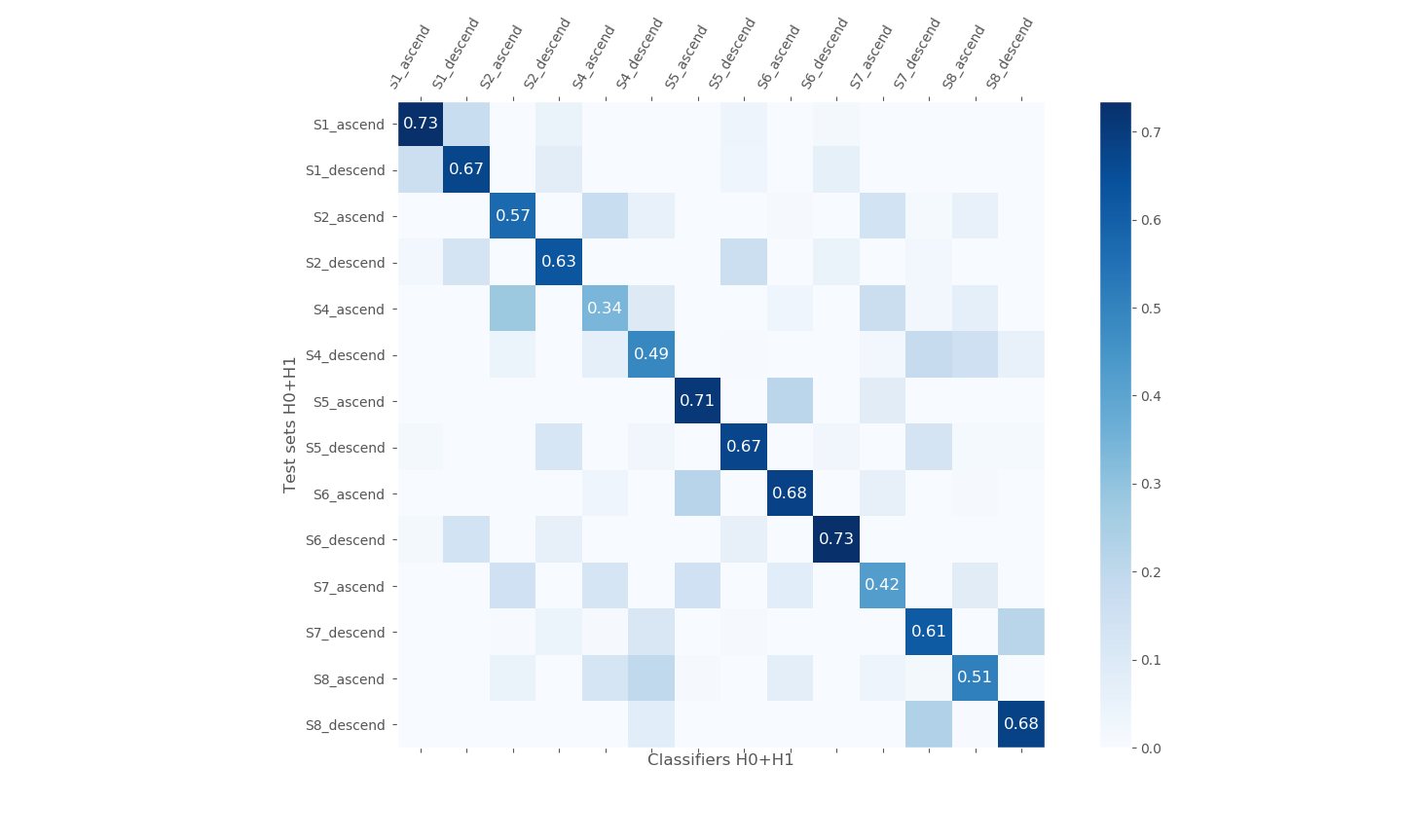}
		\includegraphics[scale=0.38,clip,trim=5cm 1.2cm 4cm 0cm]{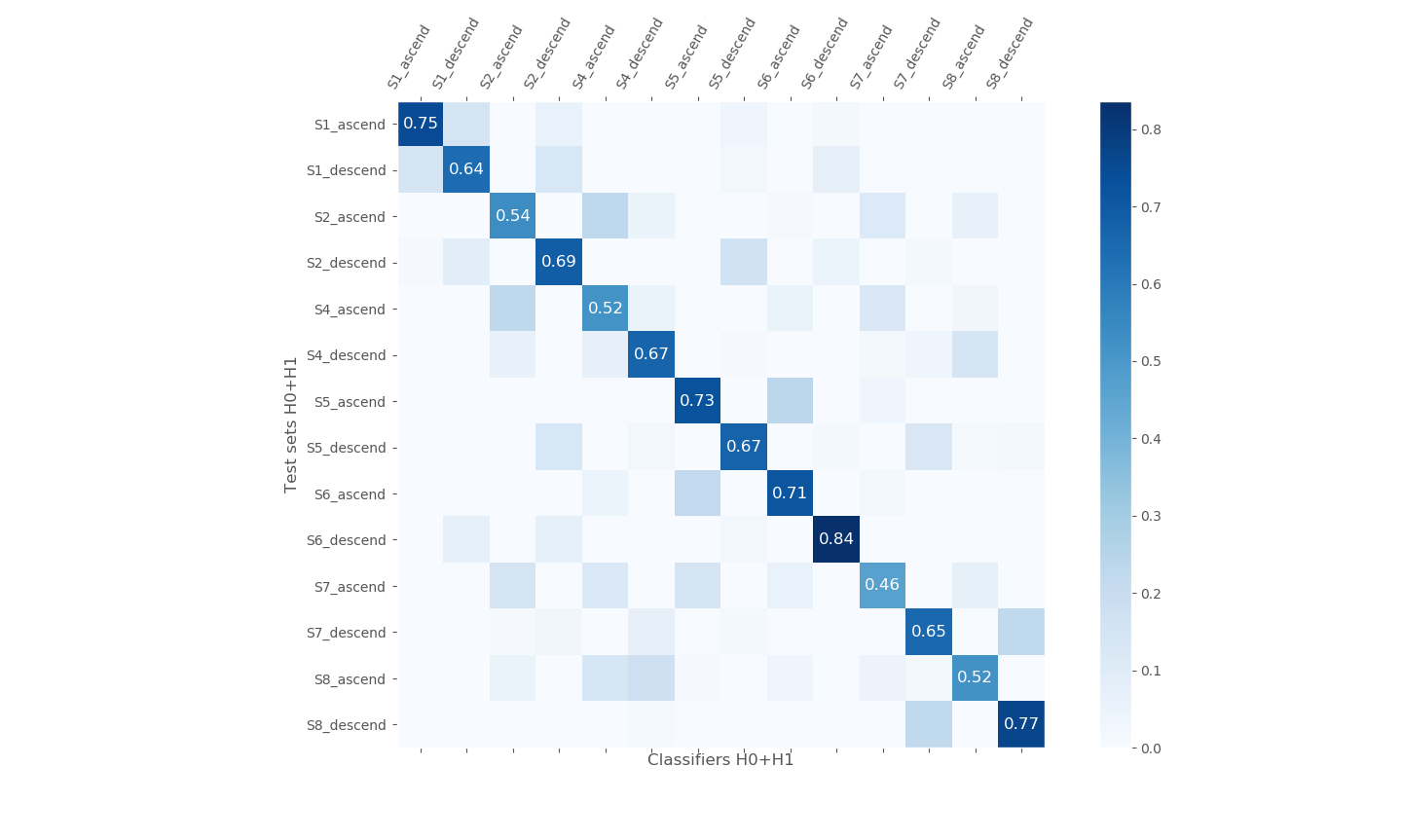}
	\end{minipage}
	\caption{Confusion matrix for the classification of ascending and descending stairs activities with standard contour (left). Confusion matrix for the classification of ascending and descending stairs activities with contour coming from density function of Figure \ref{fig_activities_density_contour} (right).}
	\label{fig_activities_confusions}
\end{figure}

\begin{figure}
	\begin{minipage}{0.99\textwidth}
		\includegraphics[width=0.49\textwidth]{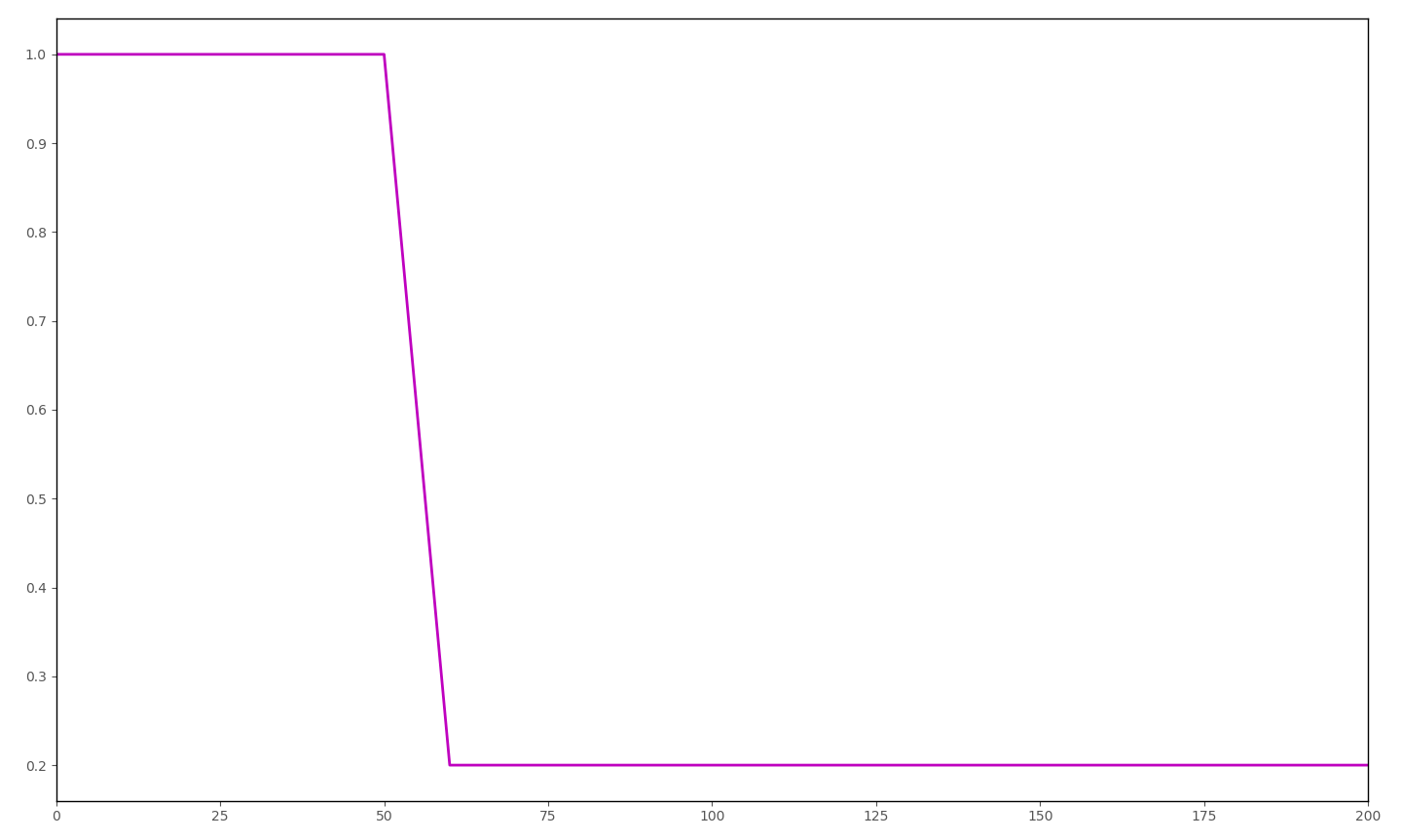}
		\includegraphics[width=0.49\textwidth]{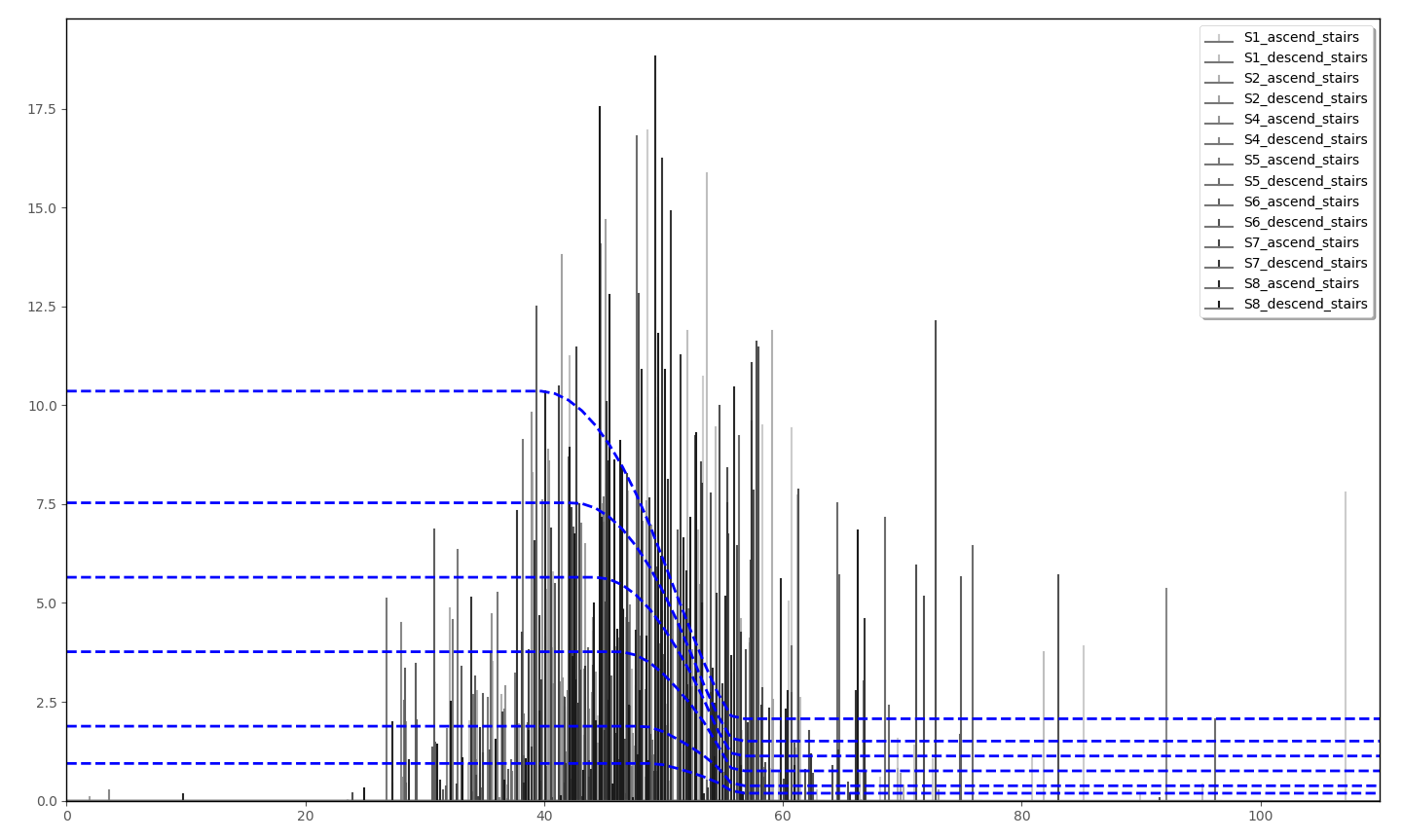}
	\end{minipage}
	\caption{Density function used for $H_1$ stable ranks in the activities classification (left) and contour lines and persistence stem plots for single data sets (right).}
	\label{fig_activities_density_contour}
\end{figure}

\subsection{Choosing the contour}
The examples above illustrate how one does data analysis by choosing a more optimal contour which gives a pseudometric on $\text{Tame}([0,\infty),\text{Vec}_K)$. In the examples this choice was made by visually inspecting stem plots and contour lines. The next step of our pipeline is to automate this process. This is the central reason behind contours and induced metrics: we want to optimize over the space of metrics to find more distinguishing invariants for objects in $\text{Tame}([0,\infty),\text{Vec}_K)$. This leads to optimizing in an appropriate function space since contours arise from density functions (see Section \ref{generalized_persistence and persistence contours}). Another way would be to represent densities by functions whose shape is controlled by few parameters, such as a beta distribution with shape parameters \(\alpha\) and \(\beta\). The optimization then reduces to these parameters.

\appendix
\section{Proofs of Propositions from Section~\ref{sec:hier_stab}}\label{adfdfsgfhd}
\begin{proof}[Proof of Proposition~\ref{prop:Lipschitz}]
	\noindent
	(1):\quad 
	If $d(X,Y)=\infty$, there is nothing to prove. Assume $\epsilon:=d(X,Y)<\infty$. 
	For  $t$ in $[0,\infty)$, there are  inclusions $B(Y,t)\subset B(X,t+\epsilon)$ and
	$B(X,t)\subset B(Y,t+\epsilon)$ which imply $\widehat{R}_d(Y)(t)\geq \widehat{R}_d(X)({t+\epsilon})$ and
	$\widehat{R}_d(X)({t})\geq \widehat{R}_d(Y)({t+\epsilon})$. As this happens for all $t$, we can conclude   
	$\epsilon\geq d_{\bowtie}(\widehat{R}_d(X),\widehat{R}_d(Y))$.
	\smallskip
	
	\noindent
	(2):\quad 
	Using (1), it is enough to prove that, for non-increasing functions $f$ and $g$:
	\[\text{max}\{f(0),g(0)\} d_{\bowtie} (f,g)^{\frac{1}{p}}\geq L_p(f,g).\]
	The  inequality is clear if  $d_{\bowtie}(f,g)=\infty$. Assume there is  $\epsilon$ 
	such that $f(t)\geq g(t+\epsilon)$ and $g(t)\geq f(t+\epsilon)$ for any $t$. 
	This    together with the fact that $f$ and $g$ are non-increasing  imply $h\geq f\geq h_\epsilon$ and $h\geq g\geq h_\epsilon$ where  
	$h=\text{max}\{f,g\}$ and $h_\epsilon$ is the function $t\mapsto h(t+\epsilon)$.  
	The desired inequality is then a consequence of  $h$ being non-increasing and
	the fact   $(b-a)^p\leq b^p-a^p$  for $a\leq b$ in $[0,\infty)$ and $p\geq 1$ which give:
	\[L_p(h,h_\epsilon)^p=\int_0^\infty (h(t)- h_\epsilon(t))^p dt\leq 
	\int_0^\infty h(t)^p- h_\epsilon(t)^p dt=\]
	\[=
	\int_{0}^{\epsilon}h(t)^pdt\leq h(0)^p\epsilon=\text{max}\{f(0),g(0)\}^p\epsilon.\]
\end{proof}

\begin{proof}[Proof of Proposition~\ref{sgdsths}]
	Let  $\{d_\alpha\}_{\alpha\in[0,\infty]}$   be a non-decreasing sequence of pseudometrics on $T$. Choose $\epsilon $ in $(0,\infty)$. For $\alpha$ in  $(0,\infty)$, set $\lfloor{\alpha}/{\epsilon}\rfloor$ to be the biggest natural number not bigger than ${\alpha}/{\epsilon}$. Define $d_\alpha^{\epsilon}:=d_{\lfloor{\alpha}/{\epsilon}\rfloor\epsilon }$ and $d_\infty^{\epsilon}:=d_{\infty}$. In this way, any $\epsilon$ leads to a new non-decreasing sequence $\{d_\alpha^{\epsilon}\}_{\alpha\in[0,\infty]}$ of pseudometrics on $T$. This new sequence is also  non-decreasing. Let $\overline{R}_\epsilon(X)\colon [0,\infty)^2\to [0,\infty)$ be the function  corresponding to this new sequence as defined in~\ref{asfgdsfghs}.(2). Since  $\{d_\alpha^{\epsilon}\}_{\alpha\in[0,\infty]}$ is constant on intervals of the form $[n\epsilon, (n+1)\epsilon)$ where $n$ is a natural number, the function   $\overline{R}_\epsilon(X)\colon [0,\infty)^2\to [0,\infty)$ is  Lebesgue measurable as it is constant  on left closed rectangles that cover  $[0,\infty)^2$. Note that $\overline{R}(X)$ is the limit of  $\overline{R}_\epsilon(X)$ as $\epsilon$ goes to $0$. As a limit of measurable functions, $\overline{R}(X)$ is then  also measurable.
\end{proof}

\begin{proof}[Proof of Proposition~\ref{afgsfhg}]
	Since to prove this proposition one   can use exactly the same   arguments as in the proof of Proposition~\ref{prop:Lipschitz}, we illustrate how to show statement (1) only. 
	If $d_{\infty}(X,Y)=\infty$, then the statement is clear. Assume $\epsilon:=d_{\infty}(X,Y)<\infty$. 
	Since  $\{d_\alpha\}_{\alpha\in[0,\infty]}$ is non-decreasing, for any $(\alpha,t)$ in $[0,\infty)\times [0,\infty)$, we have an inclusion
	$B_{d_{\alpha}}(Y,t)\subset  B_{d_{\alpha}}(X,t+\epsilon)$ which yields
	$\overline{R}(Y)(\alpha,t)\geq \overline{R}(X)(\alpha ,t+\epsilon)$.
	By symmetry  also $\overline{R}(X)(\alpha,t)\geq \overline{R}(Y)(\alpha,t+\epsilon)$, and hence 
	$\epsilon \geq d_{\bowtie}\left(\overline{R}(X),\overline{R}(Y)\right)$.
\end{proof}

\bibliographystyle{plain}
\bibliography{references}
\end{document}